\documentclass{amsart}
\usepackage{latexsym,amsxtra,amscd,ifthen}
\usepackage{amsfonts}
\usepackage{pifont}
\usepackage{verbatim}
\usepackage{amsmath}
\usepackage{graphicx}
\usepackage{amsthm}
\usepackage{amssymb}
\usepackage{url,color}
\usepackage[arrow,matrix]{xy}
\usepackage{pict2e,epic,eepic,pstricks}

\usepackage{anysize}\marginsize{25mm}{25mm}{30mm}{30mm}
\allowdisplaybreaks[4]
\numberwithin{equation}{section}
\numberwithin{table}{section}
\numberwithin{equation}{section}
\newtheorem{theorem}{Theorem}[section]
\newtheorem{lemma}[theorem]{Lemma}
\newtheorem{question}[theorem]{Question}
\newtheorem{proposition}[theorem]{Proposition}

\newtheorem{corollary}[theorem]{Corollary}
\theoremstyle{definition}
\newtheorem{example}[theorem]{Example}
\newtheorem{remark}[theorem]{Remark}

\theoremstyle{definition}
\newtheorem{definition}[theorem]{Definition}

\usepackage{bbm}
\usepackage{float}

\def\la{\lambda}
\def \ot{\otimes}

\def \As{\mathcal{A}ss}
\def \Com{\mathcal{C}om}

\def \ip{{\mathcal{P}}}
\def \iq{\mathcal {Q}}
\def \iu{\!{\it{\Upsilon}}}

\def\gkdim{\operatorname{GKdim}}

\def \Ker{\operatorname{Ker}}

\def \ker{\operatorname{Ker}}

\def \1{\mathbbm 1}
\def \C{\mathcal{C}}

\def \k{\Bbbk}

\def \Mag{\mathcal{M}ag}

\def \S{\mathbb{S}}

\def \n{[n]}
\newcommand{\up}[1]{{^{#1}{\it \!\Upsilon}}}

\newcommand{\ucr}[1]{\underset{#1}{\circ}}
\def\rmod{\mbox{\rm mod-}}
\def\Op{\mathsf{Op}}
\def\hot{\underset{{\rm H}}{\otimes}}
\newcommand{\Ar}{\operatorname{Ar}}

\begin{document}
\title[2-unitary Operads of GK-dimension 3]
{2-unitary Operads of GK-dimension 3}

\subjclass[2010]{}


\author{Yan-Hong Bao}
\address{(Bao) School of Mathematical Sciences, Anhui University, Hefei 230601, China}
\address{(Bao) Anhui University Center for Pure Mathematics, Hefei 230601, China}
\email{baoyh@ahu.edu.cn, yhbao@ustc.edu.cn}

\author{Dong-Xing Fu}
\address{(Fu) School of Mathematical Sciences, Anhui University, Hefei 230601, China}
\email{fudongxing920811@126.com}

\author{Yu Ye}
\address{(Ye) School of Mathematical Sciences,
University of Sciences and Technology of China, Hefei 230001, China}
\email{yeyu@ustc.edu.cn}

\author{James J. Zhang}
\address{(Zhang) Department of Mathematics, Box 354350,
University of Washington, Seattle, Washington 98195, USA}
\email{zhang@math.washington.edu}

\date{\today}

\begin{abstract}
We study and classify the 2-unitary operads of Gelfand-Kirillov
dimension three.
\end{abstract}

\maketitle

\setcounter{section}{-1}

\section{Introduction}
\label{xxsec0}

Algebraic operads originated from homotopy theory in algebraic
topology, and was first introduced by Boardman-Vogt \cite{BV}
and May \cite{Ma} in 1960s-1970s. During the recent 20 years,
operad theory has become an important tool in homological algebra,
category theory, algebraic geometry and mathematical physics.
It is well-known that every operad encodes an algebra system. For
example, $\As$ encodes all unital associative algebras. Further,
a $\Bbbk$-linear operad itself is an algebraic object similar to
an associative algebra, and algebraic structures of operads have
been widely investigated by many mathematicians, see
\cite{BYZ, Dot, DK, DMR, DT, Fr1,Fr2,KP,LV, MSS, QXZZ}.

The Gelfand-Kirillov dimension of an associative algebra is a useful 
numerical invariant in ring theory and noncommutative algebraic 
geometry, see \cite{KL}. In a similar way, the Gelfand-Kirillov 
dimension can be defined for other algebraic objects including 
algebraic operads \cite{BYZ, Fi}. Let $\Bbbk$ is a base field. An 
operad $\ip$ is said to be {\it locally finite} if each $\ip(n)$ is 
finite dimensional over $\Bbbk$. In this paper we only consider 
locally finite operads. The {\it Gelfand-Kirillov dimension} (or 
{\it GK-dimension} for short) of a locally finite operad $\ip$ 
is defined to be
\[\gkdim\ip\colon=\limsup_{n\to \infty}
\log_n\left(\sum_{i=0}^n \dim_{\Bbbk} \ip(i)\right).\]
We refer to \cite{BYZ, KP, QXZZ} for more information related to
the GK-dimension of an operad.

Recall that an operad $\ip$ is \textit{unitary} if $\ip(0)=
\Bbbk \1_0$ with a basis element $\1_0$ (which is called a
$0$-unit), see \cite[Section 2.2]{Fr2}. Denote by $\Op_+$ the
category of unitary operads, in which a morphism preserves the
0-unit. A \textit{2-unitary operad} $\ip$ is a unitary operad
$\ip$ equipped with a morphism $\Mag\to \ip$ in $\Op_+$, where
$\Mag$ is the unital magmatic operad (see \cite[Section 8.4]{BYZ}
or \cite[Section 4.1.10]{Lo}). The definition of a $2a$-unitary
operad is given in Definition \ref{xxdef1.5}(4). In \cite{BYZ},
the authors proved that GK-dimension of a 2-unitary operad $\ip$
is either an nonnegative integer or infinity and that the generating
series of $\ip$ is a rational function when $\gkdim \ip<\infty$.
The pattern of GK-dimension of a general non-2-unitary operad (or
nonsymmetric operad) is very different, see Remark \ref{xxrem5.5}.

The only 2-unitary operad of GK-dimension 1 is $\Com$ that encodes
all unital commutative algebras. All locally
finite 2-unitary operads of GK-dimension $2$ were classified
in \cite[Theorem 0.6]{BYZ}. One way of viewing this classification
is the following. We refer to \cite[Section 6]{BYZ} for the
construction of $2$-unitary operads of GK-dimension $2$.

\begin{theorem} \cite[Theorem 0.6]{BYZ}
\label{xxthm0.1}
There are natural equivalences between
\begin{enumerate}
\item[(1)]
the category of finite dimensional, not necessarily unital,
$\Bbbk$-algebras,
\item[(2)]
the category of $2$-unitary operads of GK-dimension $\leq 2$,
\item[(3)]
the category of $2a$-unitary operads of GK-dimension $\leq 2$.
\end{enumerate}
\end{theorem}

In the case of GK-dimension 3, our result is not 
as clean as Theorem \ref{xxthm0.2}. Nevertheless, we will 
provide a classification. Recall that an operad is called 
{\it $\Com$-augmented} if there is an operadic unit map
$u_{\ip}: \Com\to \ip$. The morphisms in the category of
$\Com$-augmented operads are supposed to be compatible with
operadic unit maps. Here is the main result of this paper.

\begin{theorem}
\label{xxthm0.2}
There is a natural equivalence between
\begin{enumerate}
\item[(1)]
the category of finite dimensional trident algebras,
\item[(2)]
the category of $\Com$-augmented operads of GK-dimension $3$.
\end{enumerate}
\end{theorem}

If ${\text{char}}\; \Bbbk\neq 2$, we also prove that every
2-unitary operad of GK-dimension 3 is equipped with a
$\Com$-augmentation with possibly new 2-unit [Proposition
\ref{xxpro2.5}]. Combining Theorem \ref{xxthm0.2} with
Proposition \ref{xxpro2.5}, we obtain a classification of
2-unitary operads of GK-dimension 3 in terms of finite dimensional 
trident algebras. However, Theorem \ref{xxthm0.2} fails if
the condition ``$\Com$-augmented'' in part (2) is replaced by 
``2-unitary'' [Example \ref{xxex5.4}]. 

The definition of a trident algebra is given in Section
\ref{xxsec3}. Roughly speaking, a trident algebra consists
of a pair of $\Bbbk$-vector spaces $(A,M)$ equipped with some
algebraic structures and a pair of $\Bbbk$-linear maps $(f,g)$
satisfying some identities. Note that a seemingly technical
result, Theorem \ref{xxthm0.2}, makes some questions easy to
solve. For example, we prove

\begin{corollary}[Proposition \ref{xxpro5.3}]
\label{xxcor0.3}
There is no $\Com$-augmented Hopf operad of GK-dimension 2.
\end{corollary}

Corollary \ref{xxcor0.3} motivates the following question.

\begin{question}
\label{xxque0.4}
Is there a $\Com$-augmented Hopf operad of finite GK-dimension 
large than $2$.
\end{question}

This paper is organized as follows. In Section \ref{xxsec1} we 
recall some basic definitions and properties of 2-unitary operads. 
We prove some properties of 2-unitary operads of GK-dimension 3 
in Section \ref{xxsec2}. A key preliminary result is that a 
2-unitary operad of GK-dimension 3 is $\Com$-augmented after 
changing the 2-unit [Proposition \ref{xxpro2.5}]. We define 
a concept of a trident algebra in Section \ref{xxsec3}. In 
Section \ref{xxsec4}, we construct an operad from a trident 
algebra and complete the classification of 2-unitary operads of 
GK-dimension 3 by Theorem \ref{xxthm0.2} and Proposition
\ref{xxpro2.5}. In Section \ref{xxsec5}, we give some comments, 
examples, and remarks. A complete but tedious proof of Theorem 
\ref{xxthm4.1} is given in Appendix (Section \ref{xxsec6}).

\bigskip

\noindent\textbf{Acknowledgments}.
Y.-H. Bao and Y. Ye were partially supported by the Natural
Science Foundation of China (Grant Nos. 11871071, 11971449).
J.J. Zhang was partially supported by the US National Science
Foundation (Grant Nos. DMS-1700825 and DMS-2001015).

\bigskip

\section{Preliminaries}
\label{xxsec1}

Throughout let $\k$ be a base field, and every object is over 
$\Bbbk$. Let $n$ be a nonnegative integer. Set 
$[n]=\{1, 2, \cdots, n\}$ for $n>0$ and $[0]=\varnothing$. We 
use $\S_n$ to denote the symmetric group for $n\ge 0$. By
convention, $\S_0$ is the trivial group. Following the notation
introduced in \cite[Section 8.1]{BYZ}, for each $\sigma\in \S_n$,
we use the sequence  $(i_1, i_2, \cdots, i_n)$ to denote a
permutation $\sigma\in \S_n$ with $\sigma(i_k)=k$ for all
$1\le k\le n$. Denote by $\S$ the disjoint union of all symmetric
group $\S_n$ for all $n\ge 0$. Recall that a $\Bbbk\S$-module 
(or $\S$-module) means a sequence $\{\ip(n)\}_{n\ge 0}$ of right 
$\Bbbk\S_n$-modules, where the right $\S_n$-action on $\ip(n)$ 
is denoted by $\ast$.

In this section, we retrospect some basic facts about operads.

\subsection{Definitions}
\label{xxsec1.1}
From different viewpoints, there are various definitions about
operads. In this paper, we mainly use the {\it partial} definition
and refer to \cite[Chapter 5]{LV} for other versions of the
definition.

\begin{definition} \cite[Section 2.1]{Fr2}
\label{xxdef1.1}
An \emph{operad} $\ip$ consists of the following data:
\begin{enumerate}
\item[(i)]
a $\Bbbk \S$-module $\{\ip(n)\}_{n\geq 0}$, where an element
in $\ip(n)$ is called an \textit{$n$-ary operation}.
\item[(ii)]
an element $\1\in \ip(1)$, which is called the \emph{identity},
\item[(iii)]
for all $m\ge 1, n\ge 0$ and $1\le i\le m$, a
\textit{partial composition}:
\[-\underset{i}{\circ}-\colon \ip(m) \otimes \ip(n) \to \ip(m+n-1)\]
\end{enumerate}
satisfying the following coherence axioms:
\begin{enumerate}
\item[(OP1)]
(Identity) for all $\theta\in \ip(M)$ and all $1\leq i\leq m$,
\[\theta\underset{i}{\circ} \1 = \theta =\1 \underset{1}{\circ} \theta;
\]
\item[(OP2)]
(Associativity) for all $\la\in \ip(l), \mu\in \ip(m)$ and
$\nu\in \ip(n)$,
\begin{align}
\label{E1.1.1}\tag{E1.1.1}
(\la  \underset{i}{\circ} \mu) \underset{i-1+j}{\circ} \nu
=\la \underset{i}{\circ} (\mu \underset{j}{\circ} \nu),
& \quad 1\le i\le l, 1\le j\le m,\\
\label{E1.1.2}\tag{E1.1.2}
(\la  \underset{i}{\circ} \mu) \underset{k-1+m}{\circ} \nu
=(\la \underset{k}{\circ}\nu) \underset{i}{\circ} \mu,
& \quad 1\le i<k\le l;
\end{align}
\item[(OP3)]
(Equivariance)
for all $\mu\in \ip(m), \nu\in \ip(n)$ and $\sigma\in \S_n$,
$\phi\in \S_m$,
\begin{align}
\label{E1.1.3}\tag{E1.1.3}
\mu  \ucr{i} (\nu \ast \sigma)=&(\mu  \ucr{i} \nu)\ast \sigma',\\
\label{E1.1.4}\tag{E1.1.4}
(\mu\ast \phi) \ucr{i} \nu=&(\mu \ucr{\phi(i)} \nu)\ast \phi'',
\end{align}
where  $\sigma'=1_m\ucr{i} \sigma$ and $\phi''=\phi\ucr{i}1_n$
are given by the partial composition in the associative algebra
operad $\As$. We refer to \cite[Section 8]{BYZ} for more details
concerning $\sigma'$ and $\phi''$.
\end{enumerate}
\end{definition}

Let $\ip$ be an operad in the sense of Definition \ref{xxdef1.1}.
Then one can define the composition map by
\begin{equation}
\label{E1.1.5}\tag{E1.1.5}
\lambda \circ (\mu_1, \cdots, \mu_n)=
(\cdots ((\lambda \underset{n}{\circ} \mu_n)
\underset{n-1}{\circ} \mu_{n-1})\underset{n-2}{\circ}
\mu_{n-2}\cdots) \underset{1}{\circ}\mu_1
\end{equation}
for all $\lambda\in \ip(n)$ and $\mu_i\in \ip$ and for
$1\leq i\leq n$ \cite[Remark 1.3]{BYZ}.

\begin{example}\cite[Section 5.2.10]{LV}
\label{xxex1.2}
Let $\Com$ denote the commutative algebra operad. The space of
$n$-ary operations of $\Com$ is $\Com(n)=\Bbbk \1_n$ equipped
with the trivial action of the symmetric group and the partial
composition is given by $\1_{m}\ucr{i} \1_n=\1_{m+n-1}$ for
all $m, n, i$. Note that $\1_1$ is the identity $\1$ of $\Com$.
\end{example}

\begin{example}\cite[Section 13.8]{LV}
\label{xxex1.3}
Suppose that $\Mag$ is the operad generated by the $\S$-module
$$(\Bbbk\mu, \Bbbk\1, \Bbbk\S_2\nu, 0, 0,\cdots)$$
and subject to  relations
$$\nu  \ucr{i} \mu = \1, (i=1, 2),$$
where $\Bbbk\S_2\nu$ is the regular $\Bbbk\S_2$-module with the basis
$\nu$. In this paper we use $\1_0$ for $\mu$ and $\1_2$ for $\nu$.
\end{example}

\begin{definition}\cite[Chapter 5]{LV}
\label{xxex1.4}
Let $\ip$ and $\ip'$ be operads. A {\it morphism} from
$\ip$ to $\ip'$ is a sequence of $\Bbbk\S_n$-homomorphism
$\gamma=(\gamma_n:\ip(n)\to \ip'(n))_{n\geq 0}$, satisfying
\[\gamma(\1)=\1'\quad {\rm and} \quad \gamma(\mu\ucr{i}\nu)
=\gamma(\mu)\ucr{i}\gamma(\nu),\]
where $\1$ and $\1'$ are identities of $\ip$ and $\ip'$,
respectively, and $\mu, \nu\in \ip$. Let $\Op$ denote the
category of operads.
\end{definition}

Next we collect some definitions given in \cite{BYZ, Fr2}.

\begin{definition}
\label{xxdef1.5}
\begin{enumerate}
\item[(1)]
An operad $\ip$ is called {\it unitary} if $\ip(0)= \Bbbk \1_0$,
where $\1_0$ is a basis of $\ip(0)$, and is called a
\textit{$0$-unit}. The category of unitary operads is denoted
by $\mathrm{Op_+}$, in which morphisms are operadic morphisms
preserving 0-units.
\item[(2)]
A unitary operad is said to be \emph{connected}, if $\ip(1)
=\Bbbk \1$ where $\1$ is the identity of $\ip$. In this case we
also use $\1_1$ for $\1$.
\item[(3)]
A \textit{2-unitary operad} is a unitary operad $\ip$ equipped
with a morphism $\Mag \to \ip$ in $\mathrm{Op_+}$, where $\Mag$
is the unital magmatic operad [Example \ref{xxex1.3}], or
equivalently, there is an element $\1_2\in \ip(2)$
(called a 2-unit) such that
\begin{equation}
\label{E1.5.1}\tag{E1.5.1}
\1_2\ucr{1} \1_0=\1(=\1_1)=\1_2\ucr{2} \1_0
\end{equation}
 where $\1_0$ is a 0-unit of $\ip$.
\item[(4)]
A \textit{2a-unitary} operad is a unitary $\ip$ equipped with
a morphism $\As \to \ip$ in $\mathrm{Op_+}$, or equivalently,
$\ip$ is 2-unitary with a 2-unit $\1_2$ satisfying
\begin{equation}
\label{E1.5.2}\tag{E1.5.2}
\1_2\underset{1}\circ \1_2=\1_2\underset{2}\circ \1_2.
\end{equation}
In this case $\1_2$ is called {\it $2a$-unit}. 
\item[(5)]
An operad $\ip$ is called {\it $\Com$-augmented} if there is a
morphism from $\Com\to \ip$. It is clear that $\Com$-augmented
operads are $2a$-unitary. In this case $\1_2\ast (2,1)=\1_2$
and $\1_2$ is called a {\it symmetric $2a$-unit}.
\end{enumerate}
\end{definition}

Let $\ip$ and $\ip'$ be 2-unitary operads. A morphism of 2-unitary
operads is a morphism $\gamma :\ip \to \ip'$ in $\mathrm{Op_+}$
satisfying the following commutative diagram
\[\xymatrix{
& \Mag\ar[dl]\ar[dr]& \\
\ip\ar[rr]^{\gamma} && \ip'
}\]
or equivalently, the operad morphism preserves the 2-unit. The
categories of 2-unitary operads, 2a-unitary operads and
$\Com$-augmented operads, are denoted by
$\Mag \downarrow \mathrm{Op_+}$, $\As \downarrow \mathrm{Op_+}$
and $\Com \downarrow \mathrm{Op_+}$, respectively.
Let $\dim$ denote $\dim_{\Bbbk}$.

\begin{definition}
\label{xxdef1.6}
Let $\ip=(\ip(n))_{n\ge 0}$ be a locally finite operad, i.e.
$\dim \ip(n)<\infty$ for all $n\ge 0$.
\begin{enumerate}
\item[(1)]
The {\it generating series} of $\ip$ is defined to be
$$G_{\ip}(t)=\sum_{n=0}^{\infty} \dim \ip(n) t^n \in {\mathbb Z}[[t]].$$
\item[(2)]
The {\it Gelfand-Kirillov dimension} ({\it GK-dimension} for short) of $\ip$
is defined to be
$$\gkdim (\ip)=\limsup_{n\to\infty} \log_{n} (\sum_{i=0}^{n} \dim \ip(i)).$$
\end{enumerate}
\end{definition}

\subsection{Truncation Ideals}
\label{xxsec1.2}

Let $\ip$ be a unitary operad and $I$ a subset of $[n]$. Recall
that a {\it restriction operator} \cite[Section 2.2.1]{Fr2} means
\begin{equation}
\label{E1.6.1}\tag{E1.6.1}
\pi^{I}\colon \ip(n)\to\ip(s),\qquad \pi^{I}(\theta)=
\theta\circ(\1_{\chi_{I}(1)}, \cdots, \1_{\chi_{I}(n)})
\end{equation}
for all $\theta\in \ip(n)$, where $\chi_{I}$ is the characteristic
function of $I$, i.e. $\chi_{I}(x)=1$ for $x\in I$ and
$\chi_{I}(x)=0$ otherwise. Note that $\circ$ is defined in
\eqref{E1.1.5}. If $I=\{i_1, \cdots, i_s\}\subset [n]$ with
$1\le i_1<\cdots<i_s\le n$, we also denote $\pi^I$ as
$\pi^{i_1, \cdots, i_s}$. We refer to \cite[Section 2.3]{BYZ}
for more details.

For integers $k\ge 1$, the $k$-th truncation ideals $\up{k}$ is
defined by
\begin{equation}
\label{E1.6.2}\tag{E1.6.2}
^k\iu_\ip(n) =
\bigcap\limits_{I\subset \n,\, |I|\le k-1} \ker\pi^I
=\begin{cases}
\bigcap\limits_{I\subset \n,\, |I|= k-1} \ker\pi^I,
   & \text{if }n\ge k;\\
\quad\ \ 0, & \text{otherwise}.
\end{cases}
\end{equation}
By convention, let ${^0\iu_\ip}=\ip$. If no confusion, we
write ${^k\iu}={^k\iu_\ip}$ for brevity.

For every subset $I=\{i_1, \cdots, i_s\}\subset [n]$ with
$i_1<\cdots<i_s$, we denote a permutation
\[c_I\colon=(i_1, \cdots, i_s, 1, \cdots, i_1-1, i_1+1,
\cdots, i_s-1, i_s+1, \cdots, n)\in \S_n.\]

\begin{theorem}\cite[Theorem 4.6]{BYZ}
\label{xxthm1.7}
Let $\ip$ be a $2$-unitary operad. For each $k\ge 0$, let
\[\Theta^k\colon =\{\theta_1^k, \cdots, \theta_{z_k}^k\}\]
be a $\Bbbk$-basis for $\up{k}(k)$. Let $\textbf{B}_k(n)$
be the set
\[\{\1_2\circ (\theta_i^k, \1_{n-k})\ast c_I
\mid 1\le i\le z_k, I\subset [n], |I|=k.\}\]
Then $\ip(n)$ has a $\Bbbk$-basis
\[\bigcup\limits_{0\le k\le n} \textbf{B}_k(n)
=\{\1_n\}\cup \bigcup\limits_{1\le k\le n} \textbf{B}_k(n),\]
and for every $k\ge 1$, $\up{k}(n)$ has a $\Bbbk$-basis
$\bigcup_{k\le i\le n} \textbf{B}_i(n)$.
\end{theorem}

\begin{lemma}\cite[Lemma 5.2]{BYZ}
\label{xxlem1.8}
Let $\ip$ be a $2$-unitary operad and $f_\ip(k)=\dim \up{k}(k)$
for each $k\ge 0$. Then
\begin{enumerate}
\item[(1)]
$G_\ip(t)=\sum\limits_{k=0}^\infty f_\ip(k)\dfrac{t^k}{(1-t)^{k+1}}$.
\item[(2)]
$\gkdim \ip=\max\{k\mid f_\ip(k)\neq 0\}+1=\min\{k\mid \up{k}=0\}$.
\end{enumerate}
\end{lemma}

Combining Lemma \ref{xxlem1.8}(2) with \cite[Proposition 0.5]{BYZ}, 
if $\ip$ is $2$-unitary, then there is a canonical morphism
of $2$-unitary operads
\begin{equation}
\label{E1.8.1}\tag{E1.8.1}
\epsilon: \quad
\ip \longrightarrow \ip/{^1\iu} =\Com.
\end{equation}

\section{Basic Facts of 2-unitary Operads of GK-dimension 3}
\label{xxsec2}

Let $\ip$ be a 2-unitary operad of GK-dimension 3. By Theorem
\ref{xxthm1.7} and Lemma \ref{xxlem1.8}, we have the following
basic facts:

\begin{enumerate}
\item[(1)]
${^2\iu}\neq 0$, and ${^k\iu}=0$ for all $k\ge 3$.
\item[(2)]
$\dim \ip(n)=1+dn+m\dfrac{n(n-1)}{2}$, where
$d:=f_{\ip}(1)=\dim {^1\iu}(1)$ and $m:=f_\ip(2)=\dim {^2\iu}(2)$.
\item[(3)]
$G_\ip(t)=\dfrac{1}{1-t}+d\dfrac{t}{(1-t)^2}+m\dfrac{t^2}{(1-t)^3}$.
\end{enumerate}

Based on the above facts, we have the following lemmas, which is
useful to understand the structure of a 2-unitary operad of
GK-dimension 3.

\begin{lemma}
\label{xxlem2.1}
Let $\ip$ be a $2$-unitary operad of GK-dimension $3$ and $\1_2$ a $2$-unit.
Then for all $\tau, \mu\in \up{2}(2)$,
\begin{align}
\tau \ucr{1} \1_2=& \1_2\ucr{2} \tau+(\1_2\ucr{2} \tau)\ast (12),
\label{E2.1.1}\tag{E2.1.1}\\
\tau \ucr{2} \1_2=& \1_2\ucr{1} \tau+(\1_2\ucr{1} \tau)\ast (23).
\label{E2.1.2}\tag{E2.1.2}\\
\tau \ucr{i} \mu=& 0 \quad (i=1, 2).
\label{E2.1.3}\tag{E2.1.3}
\end{align}
\end{lemma}

\begin{proof} By a direct calculation, we have
\[(\tau \ucr{1} \1_2-\1_2\ucr{2} \tau
  -(\1_2\ucr{2} \tau)\ast (12))\ucr{i}\1_0=0\]
for all $1\le i\le 3$. It follows that
\[\tau \ucr{1} \1_2-\1_2\ucr{2} \tau
  -(\1_2\ucr{2} \tau)\ast (12)\in \up{3}(3).\]
Since $\ip$ is of GK-dimension 3 and $\up{3}(3)=0$,
Equation \eqref{E2.1.1} holds. Similarly, \eqref{E2.1.2} and 
\eqref{E2.1.3} hold.
\end{proof}

Let $\ip$ be a $2$-unitary operad with a $2$-unit $\1_2$. By
convention, we define $\1'_n=\1_n$ for $n=0,1,2$.  Recall from
\cite[Section 2]{BYZ} that, for every $n\ge 3$, one can
define inductively that
\[\1_n=\1_2\ucr{1} \1_{n-1}, \ \
{\rm and}\ \ \1'_n=\1_2\ucr{2} \1'_{n-1}.\]
By Definition \ref{xxdef1.5}(4) a 2-unitary operad is called
2a-unitary if $\1_2$ is associative, or equivalently $\1_3=\1'_3$.

\begin{lemma}
\label{xxlem2.2}
Let $\ip$ be a 2-unitary operad of GK-dimension 3 with 2-unit $\1_2$. 
Then $\1_2$ is a 2a-unit. Moreover, if $\1_2$ is a 2a-unit, then so is 
$\1_2+\tau$ for any $\tau\in \up{2}(2)$.
\end{lemma}

\begin{proof} Suppose that $\1_2$ is a $2$-unit of $\ip$. By 
definition, $\1_3=\1_2\ucr{1} \1_2$ and $\1'_3=\1_2\ucr{2} \1_2$.
One can easily check that $(\1_3-\1'_3) \ucr{i} \1_0=0$ for all
$i=1,2,3$. This means that $\1_3-\1'_3\in \up{3}(3)$. Since $\ip$ is of
GK-dimension 3, $\up{3}(3)=0$. Thus $\1_3=\1'_3$ as required.

Clearly,  for any $\tau\in \up{2}(2)$, we have
\[(\1_2+\mu)\ucr{i} \1_0=\1_2\ucr{i} \1_0=\1_1\]
for $i=1, 2$. So $\1_2+\mu$ is a 2-unit. Moreover, 
by Lemma \ref{xxlem2.1} \eqref{E2.1.3}, we have
\begin{align*}
(\1_2+\tau)\ucr{i} (\1_2+\tau)=\1_2\ucr{i}\1_2+\tau\ucr{i}\1_2+\1_2\ucr{i}\tau
\end{align*}
for $i=1, 2$. Since $\1_2$ is a 2a-unit, $\1_2\ucr{1}\1_2=\1_2\ucr{2}\1_2$.
By direct computation, we have
\[(\1_2\ucr{1}\tau +
\tau\ucr{1}\1_2-\1_2\ucr{2}\tau-\tau\ucr{2}\1_2)\ucr{i} \1_0=\tau-\tau=0\]
for $i=1, 2, 3$.
Therefore, 
$(\1_2\ucr{1}\tau +\tau\ucr{1}\1_2-\1_2\ucr{2}\tau-\tau\ucr{2}\1_2)
\in \up{3}(3)$. Since $\ip$ is of GK-dimension $3$ and $\up{3}(3)=0$, 
we know $\1_2\ucr{1}\tau +\tau\ucr{1}\1_2=\1_2\ucr{2}\tau +\tau\ucr{2}\1_2$.
It follows that $\1_2+\tau$ is a 2a-unit.
\end{proof}

\begin{lemma}\cite[Lemma 2.7]{BYZ}
\label{xxlem2.3}
Let $\mathcal{P}$ be a $2a$-unitary operad. Then the following hold.
\begin{enumerate}
\item[$(1)$]
For every $n\ge 3$, $\1_n=\1'_n$.
\item[$(2)$]
For every $n\ge 1$ and $k_1, \cdots, k_n\ge 0$,
$\1_n\circ (\1_{k_1}, \cdots, \1_{k_n})=\1_{k_1+\cdots+k_n}$.
\end{enumerate}
\end{lemma}

\begin{lemma}
\label{xxlem2.4}
Let $\ip$ be a $2$-unitary operad of GK-dimension $3$. Suppose 
that $\up{1}(1)$ has a $\Bbbk$-basis $\{\delta_j\mid j\in [d]\}$
and $\up{2}(2)$ has a $\Bbbk$-basis $\{\tau_s\mid s\in [m]\}$.
Then for every $n\ge 3$, $\ip(n)$ has a $\Bbbk$-basis:
\begin{align}
\label{E2.4.1}\tag{E2.4.1}
\{\1_n\}\cup \{\delta^n_{(i), j}\mid i\in [n], j\in [d]\}
\cup \{\tau^n_{(i_1i_2), s}\mid 1\le i_1<i_2\le n, s\in [m]\},
\end{align}
where $\delta^n_{(i), j}=(\1_n\ucr{1} \delta_j)\ast c_{i}$,
$\tau^n_{(i_1i_2), s}=(\1_{n-1}\underset{1}{\circ}\tau_s)\ast
c_{i_1i_2}$,
and $c_i=(i, 1, \cdots, i-1, \hat{i}, i+1, \cdots, n)$ and
$c_{i_1i_2}=(i_1, i_2, 1, \cdots, i_1-1, \hat{i_1}, i_1+1,
\cdots, i_2-1, \hat{i_2}, i_2+1, \cdots, n)$.
\end{lemma}

\begin{proof}
Since $\ip$ is a $2$-unitary operad of GK-dimension $3$, we have
$\up{k}(k)=0$ for all $k\ge 3$. By Theorem \ref{xxthm1.7}, we can
choose a basis of $\ip(n)$ as follows
\[\{\1_n\}\cup\{\1_2\circ (\delta_j, \1_{n-1})\ast c_i
\mid i\in [n], j\in [d]\}\cup
\{\1_2\circ (\tau_s, \1_{n-2})\ast c_{i_1i_2}
\mid 1\le i_1<i_2\le n, s\in [m]\}.\]
By Lemmas \ref{xxlem2.2} and \ref{xxlem2.3}
\begin{align*}
\1_n\ucr{1} \delta_j=\1'_n \ucr{1} \delta_j
&=(\1_2\ucr{2}\1_{n-1}) \ucr{1} \delta_j
=(\1_2 \ucr{1} \delta_j)\ucr{2}\1_{n-1}=\1_2\circ (\delta_j, \1_{n-1}), \\
\1_{n-1}\ucr{1} \tau_s=\1'_{n-1} \ucr{1} \tau_s
&=(\1_2\ucr{2}\1_{n-2})\ucr{1} \tau_s
=(\1_2 \ucr{1} \tau_s)\ucr{3}\1_{n-2}=\1_2\circ (\tau_s, \1_{n-2}),
\end{align*}
we immediately obtain basis \eqref{E2.4.1} of $\ip(n)$.
\end{proof}

\begin{proposition}
\label{xxpro2.5}
Suppose ${\text{char}}\; \Bbbk\neq 2$. Let $\ip$ be a 2-unitary 
operad of GK-dimension 3 with a $2$-unit $\1_2$. Let
$$ \1'_2:=\dfrac{1}{2}(\1_2+\1_2\ast (12)).$$
Then $\1'_2$ is also a $2a$-unit. Consequently, $\ip$ is 
$\Com$-augmented.
\end{proposition}

\begin{proof} It is easy to check that $\1'_2$ is a 2-unit. By
Lemma \ref{xxlem2.2}, $\1'_2$ is a $2a$-unit, namely,
$(\ip, \1_0,\1,\1'_2)$ is a 2a-unitary operad.

After replacing $\1_2$ by $\1'_2$ we may assume that $\1_2 \ast
(12)=\1_2$. It follows from induction and Lemma \ref{xxlem2.3}(1)
that $\1_n \ast \sigma=\1_n$ for all $\sigma\in \S_n$. Therefore
there is a canonical morphism from $\Com$ [Example \ref{xxex1.2}]
to $\ip$ sending $\1_n\mapsto \1_n$ for all $n\geq 0$.
\end{proof}

\section{Trident Algebras}
\label{xxsec3}

Let $R$ be a untial associative algebra over $\Bbbk$ with a right 
action of an abelian group $G$ satisfying $(ab).g=(a.g)(b.g)$ and 
for all $a, b\in R$ and $g\in G$. Such an $R$ is called a 
{\it $\Bbbk G$-module algebra}. Recall that the {\it skew group 
algebra} $R\# G$ is the vector space $R\otimes \Bbbk G$ with the 
multiplication
\[(a\# g)(b\# h)= ((a.h)b)\#(gh)).\]
Furthermore, a right module $M$ over $R\# G$ means that $M$ is a 
right $R$-module and a right $\Bbbk G$-module satisfying 
$(\mu a)g=(\mu g)(a.g)$ for all $\mu\in M, g\in G$ and $a\in R$.
The following lemma is easy.

\begin{lemma}
\label{xxlem3.1}
Let $G$ be an abelian group. Let $A$ and $B$ be $G$-module algebras. 
Suppose $M, N$ are right modules over the skew groups algebras 
$A\# G$ and $B\# G$, respectively. Then $M\otimes N$ is a right 
$(A\otimes B)\# G$-module with the action given by 
\[(x\otimes y)(a\otimes b\# g)\colon= (x.(a\#g))\otimes (y. (b\#g))\]
for $x\in M, y\in N, a\in A, b\in B, g\in G$.
\end{lemma}

Let $A$ be a unital associative algebra. Clearly, the tensor 
product algebra $A\otimes A$ (also denoted by $A^{\otimes 2}$) 
admits a natural right $\S_2$-action given by 
$(a\otimes b)(12)\colon=b\otimes a$. So we obtain a skew group
algebra $(A\otimes A)\# \S_2$ (also denoted by 
$A^{\otimes 2} \# \S_2$). Let $M$ be a left $A$- right 
$A^{\otimes 2}\#\S_2$-bimodule. Equivalently, $M$ is both a right 
$\Bbbk\S_2$-module and a left $A$- right $(A\otimes A)$-bimodule
satisfying
\begin{align}
\label{E3.1.1}\tag{E3.1.1}
a (\mu (12))=& (a\mu)(12),\\
\label{E3.1.2}\tag{E3.1.2}
(\mu.(a\otimes b)) (12)=& (\mu(12)).(b\otimes a),
\end{align}
for all $a, b\in A$, $\mu\in M$.

\begin{remark}
\label{xxrem3.2}
Observe that if $M$ is both a right $\Bbbk\S_2$-module and a 
right $A$-module with the action
$M \otimes A \to M, (\mu, a)\mapsto \mu\underset{1}{\cdot} a$,
then $M$ admits another right $A$-module action given by
$$\mu\underset{2}{\cdot}a=((\mu.(12)) \underset{1}{\cdot} a).(12),$$
which is called the \textit{congruence action}. Therefore, a right
$A^{\otimes 2}\#\S_2$-module action on $M$ is equivalent to a right 
$\Bbbk \S_2$-action together with a right
$A$-module action satisfying
\begin{equation}
\label{E3.2.1}\tag{E3.2.1}
(\mu \underset{1}{\cdot}a)\underset{2}{\cdot} b 
=(\mu \underset{2}{\cdot}b)\underset{1}{\cdot} a
\end{equation}
for all $\mu\in M, a\in A$.
\end{remark}

\subsection{Tridents}
\label{xxsec3.1}
In this subsection we introduce a new algebraic system.

Let $A=\Bbbk \1_1 \oplus \bar A$ be an augmented algebra with
augmentation ideal $\bar A$. Obviously, the right regular 
$A^{\otimes 2}$-module $A\otimes A$ with an action of $\S_2$ 
given by $(a\otimes b)(12)=b\otimes a$ admits a right 
$A^{\otimes 2}\# \S_2$-module structure. Furthermore, its quotient 
module $(A\otimes A)/(\bar A \otimes \bar A)$ admits an 
$(A, A^{\otimes 2}\# \S_2)$-bimodule structure, where the left 
$A$-action is given by
\[a\cdot [1_A\otimes 1_A]\colon=[a\otimes 1_A]+[1_A\otimes a],
\quad
a\cdot [b\otimes 1_A]\colon=[(ab)\otimes 1_A],
\quad a\cdot [1_A\otimes b]\colon=[1_A\otimes (ab)]\]
for $a, b\in \bar A$,
where $[x\otimes y]$ denotes the element 
$x\otimes y+\bar A\otimes \bar A\in (A\otimes A)/(\bar A
\otimes \bar A)$ for $x\otimes y\in A \otimes A$.

In fact, $(A\otimes A)/(\bar A\otimes \bar A)$ is isomorphic to
$\Bbbk(1_A\otimes 1_A)\oplus (\bar A\otimes \Bbbk 1_A)
\oplus (\Bbbk 1_A \otimes \bar A)$ as a vector space.

Let $A$ be an augmented algebra with the augmented ideal 
$\bar A$ and $M$ an $(A, A^{\otimes 2}\# \S_2)$-bimodule.
Suppose that $E$ is an extension of $M$ by 
$(A\otimes A)/(\bar A\otimes \bar A)$ as an 
$(A, A^{\otimes 2}\# \S_2)$-bimodule. Then the triple $(A, M, E)$ 
is called a \textit{trident}.

Let $(A, M, E)$ and $(A', M', E')$ be two tridents. Suppose that 
$\alpha\colon A \to A'$ is a homomorphism of augmented algebras, 
and $\beta\colon M \to M'$ is a homomorphism of 
$(A, A^{\otimes 2}\# \S_2)$-bimodules with the 
$(A, A^{\otimes 2}\# \S_2)$-actions on $M'$ given by the algebra 
homomorphisms $\alpha\colon A \to A'$ and 
$\alpha\otimes \alpha\colon A\otimes A \to A'\otimes A'$. Clearly, 
$[\alpha\otimes \alpha]\colon (A\otimes A)/(\bar A\otimes \bar A) 
\to (A'\otimes A')/(\bar A'\otimes \bar A')$ is a homomorphism of 
$(A, A^{\otimes 2}\#\S_2)$-bimodules. Then one can obtain a 
homomorphism of extensions of $(A, A^{\otimes 2}\#\S_2)$-bimodules
\begin{align*}
\begin{CD}
0 @>>> M @>>> E @>>> (A\otimes A)/(\bar A\otimes \bar A) @>>> 0\\
&& @V{\beta}VV  @V{\gamma}VV @VV{[\alpha\otimes \alpha]}V \\
0 @>>> M' @>>> E' @>>> (A'\otimes A')/(\bar A'\otimes \bar A') @>>> 0
\end{CD}
\end{align*}
which is also called a homomorphism of tridents. Consequently, 
we obtain a category $\mathcal{T}$, called the {\it trident category}.
The following lemmas follows from Lemma \ref{xxlem3.1}.

\begin{lemma}
\label{xxlem3.3}
Let $A, B$ be an associative algebras over $\Bbbk$. Suppose that 
$M$ and $N$ are $(A, A^{\otimes 2}\#\S_2)$- and 
$(B, B^{\otimes 2}\#\S_2)$-bimodules, respectively. Then $M\otimes N$ 
is $(A\otimes B, (A\otimes B)^{\otimes 2}\#\S_2)$-bimodule. 
\end{lemma}

\subsection{Trident systems}
\label{xxsec3.2} 
There is another way of introducing a trident.
Denote by $\rmod\S_2$ the category of finite dimensional right
$\Bbbk\S_2$-modules. It is well known that $V\otimes W\in \rmod\S_2$
with the diagonal action
$$(v\otimes w)\ast (12)=(v\ast (12))\otimes (w\ast (12))$$
for $V, W\in \rmod\S_2$.  Let $B$ be an associative algebra in
the category $\rmod\S_2$. Such an algebra is also called a
$\Bbbk\S_2$-module algebra (in Hopf algebra language). Sometimes
it is called an algebra with involution (but $\Bbbk$ may not be
the complex field). A right (\textit{resp}. left) module $V$
over a $\Bbbk\S_2$-module algebra $B$ means $V\in \rmod\S_2$
and the action $V\otimes B \to V$ (\textit{resp}.
$B\otimes V\to V$) is a homomorphism of $\Bbbk\S_2$-modules.

\begin{definition}
\label{xxdef3.4}
A pair $(A,M)$ with morphisms $f,g$, or equivalently, a
quadruple $(A, M, f,g)$, is called a {\it trident system} if
\begin{enumerate}
\item[(1)]
$A=\Bbbk \1_1 \oplus \bar A$ is an augmented algebra with 
the augmentation ideal $\bar A$,
\item[(2)]
$M$ is a nontrivial $(A, A\otimes A)$-bimodule in $\rmod\S_2$,
\item[(3)]
$f\colon \bar A \to M$ is a $\Bbbk$-linear map in $\rmod\S_2$
where the $\S_2$-action on ${\bar A}$ is trivial,
\item[(4)]
$g\colon \bar A\ \otimes \bar A \to M$ is a homomorphism of right
$A\otimes A$-modules in $\rmod\S_2$,
\end{enumerate}
such that the following identities hold
\begin{align}
f(ab)=& af(b)+f(a)\cdot (b\otimes 1_A)
+f(a)\cdot (1_A\otimes b)+g(a, b)+g(b, a),
\label{E3.4.1}\tag{E3.4.1}\\
f(a)\cdot(b, c)=&ag(b, c)-g(ab, c)-g(b, ac),
\label{E3.4.2}\tag{E3.4.2}
\end{align}
for all $a,b,c\in {\bar A}$.
\end{definition}

We define morphisms between trident systems as follows.
Let $(A, M, f, g)$ and $(A', M', f', g')$ be two trident systems.
A morphism $(\alpha, \beta)\colon (A, M, f, g)\to (A', M', f', g')$
is given by an algebra homomorphism $\alpha\colon A \to A'$ and a
trident $A$-module homomorphism $\beta\colon M \to M'$ such that
the following diagrams commute
\[\begin{CD}
\bar A @>f>> M \\
@V{\alpha}VV @VV{\beta}V\\
\bar A'@>f'>> M'
\end{CD}\quad\quad\quad{\rm and}\quad\quad\quad
\begin{CD}
\bar A\otimes \bar A @>g>> M \\
@V{\alpha\otimes \alpha}VV @VV{\beta}V\\
\bar A'\otimes \bar A'@>g'>> M'
\end{CD}\]
where the $(A, A\otimes A)$-bimodule actions on $M'$ is
determined by $(A',A'\otimes A')$-bimodule actions and
the algebra homomorphisms $\alpha\colon A \to A'$ and
$\alpha\otimes \alpha\colon A\otimes A \to A'\otimes A'$.

One can define a category $\C$ consisting of all
trident systems and morphisms defined above.

\begin{proposition}
\label{xxpro3.5}
Retain the above notation. The trident category is isomorphic 
to the category of trident systems.
\end{proposition}

\begin{proof} Let $(A,M,f,g)$ be a trident system. Recall that 
$A^{\otimes 2}$ is a subring of $A^{\otimes 2}\# \S_2$. Then $f$ 
and $g$ satisfy
\begin{enumerate}
\item[(1)] $f(a)=f(a)\ast (12)$,
\item[(2)] $g(a, b)=g(b, a)\ast (12)$,
\item[(3)] $f(ab)= a\cdot f(b)+f(a)\cdot(b\otimes 1_A\#(1))
+f(a)\cdot(1_A\otimes b\#(1))+g(a, b)+g(b, a)$,
\item[(4)] $f(a)\cdot (b\otimes c\#(1))=a\cdot g(b, c)-g(ab, c)-g(b, ac)$,
\end{enumerate}
for all $a, b, c\in \bar A$. Using these equations, one can 
define an extension $E$ of $M$ by $(A\otimes A)/(\bar A\otimes \bar A)$.
To be precise, $E=M\oplus (A\otimes A)/(\bar A\otimes \bar A)$ as 
a right $\S_2$-module with the $(A, A\otimes A)$-bimodule action
given by
\begin{align*}
	&a\cdot (x, \la[1_A\otimes 1_A]+[b\otimes 1_A]+[1_A\otimes c])=(ax+\la f(a)+f(a)\cdot (b\otimes 1_A)+f(a)\cdot (1_A\otimes c)\\
	& \qquad \qquad \qquad \qquad \qquad +g(b, a)+g(a, c), \la[a\otimes 1_A]+\la[1_A\otimes a]+[ab\otimes 1_A]+[1_A\otimes ac]), \\
	&(x, \la[1_A\otimes 1_A]+[b\otimes 1_A]+[1_A\otimes c])\cdot (a\otimes 1_A)=(x\cdot (a\otimes 1_A)+g(a, c), \la[a\otimes 1_A]+[ba\otimes 1_A]),\\
	&(x, \la[1_A\otimes 1_A]+[b\otimes 1_A]+[1_A\otimes c])\cdot (1_A\otimes a)=(x\cdot (1_A\otimes a)+g(b, a), \la[1_A\otimes a]+[1\otimes ca]),\\
	&(x, \la[1_A\otimes 1_A]+[b\otimes 1_A]+[1_A\otimes c])\cdot (a\otimes a')=(x\cdot (a\otimes a')+\la g(a, a')+g(ba, a')+g(a, ca'), 0)
\end{align*}
for all $\la\in \Bbbk$,  $a, a_1, a_2, b, c\in \bar A$ and $x\in M$.
It is easy to see that $(A,M,E)$ is a trident.

Conversely, given a trident $(A, M, E)$, we construct a trident system as
follows. Suppose that 
\[0\to M \to E \xrightarrow{\pi} (A\otimes A)/(\bar A\otimes \bar A)\to 0\]
is the corresponding short exact sequence of $(A, A^{\otimes 2}\#\S_2)$-bimodules. Without loss of generality, we
assume $M$ is a sub-bimodule of $E$. Fix an element $\1_2\in E$ with 
$\pi(\1_2)=[1_A\otimes 1_A]\in (A\otimes A)/(\bar A\otimes \bar A)$.
For all $a, b\in \bar A$, we define
\begin{align*}
	f(a)\colon = & a\cdot \1_2-\1_2\cdot (a\otimes 1_A\# (1))-\1_2\cdot (1_A\otimes a\#(1)),\\
	g(a, b)\colon=& \1_2\cdot (a\otimes b\#(1))
\end{align*}
in $E$. Clearly, in $(A\otimes A)/(\bar A\otimes \bar A)$, we have 
\begin{align*}
\pi(f(a))=& a\cdot[1_A\otimes 1_A]-[a\otimes 1_A]-[1_A\otimes a]=0,\\
\pi(g(a, b))=& [1_A\otimes 1_A]\cdot (a\otimes b)=0
\end{align*}
Therefore, we obtain two $\Bbbk$-linear maps 
\[f\colon \bar A\to M, \quad {\rm and} \quad g\colon \bar A\otimes \bar A \to M.\]
It can be directly checked that $(A, M, f, g)$ is a trident system. 

Since both constructions above are canonical, these defines two functors 
that are inverse to each other.  
\end{proof}

\begin{definition}
\label{xxsec3.6}
A {\it trident algebra} means either a trident $(A,M,E)$ or
a trident system $(A,M,f,g)$.
\end{definition}

\subsection{Examples}
\label{xxsec3.3} 
We give some easy examples of trident algebras.

\begin{example}
\label{xxex3.7}
Let $A$ be an augmented algebra and let $M$ be an 
$(A, A^{\otimes 2}\# \S_2)$-bimodule.
Consider the trivial extension of $M$ by 
$(A\otimes A)/(\bar A\otimes \bar A)$. Equivalently, 
$f: {\bar A}\to M$ and $g: {\bar A}^{\otimes 2} \to M$ are 
zero maps in the corresponding trident system. In this case, 
we call $(A,M,f,g)$ is called a {\it trivial} trident algebra.
\end{example}

\begin{example}
\label{xxex3.8}
This is the trident algebra corresponding to $\mathcal{D}_A \hot
\mathcal{D}_B$, where $\mathcal{D}_A$ and $\mathcal{D}_B$ are the 
2-unitary operad defined in \cite[Example 2.4]{BYZ}.

Let $A$ and $B$ be augmented algebras. 
Clearly, $A\otimes B$ is also an augmented algebra with 
${\overline{A\otimes B}}=
{\bar A}\otimes \Bbbk 1_B+\Bbbk 1_A\otimes {\bar B} 
+{\bar A}\otimes {\bar B}$.
From Subsection \ref{xxsec3.1}, we know 
$(A\otimes A)/(\bar A\otimes \bar A)$ and $(B\otimes B)/(\bar B\otimes\bar B)$
are $(A, A^{\otimes 2}\#\S_2)$-bimodule and 
$(B, B^{\otimes 2}\# \S_2)$-bimodule, respectively. By 
Lemma \ref{xxlem3.1}, we obtain an 
$(A\otimes B, (A\otimes B)^{\otimes 2}\#\S_2)$-bimodule
\[E=[(A\otimes A)/(\bar A\otimes \bar A)]
\otimes [(B\otimes B)/(\bar B\otimes\bar B)].\]
Observe that, 
\begin{align*}
E=(\Bbbk[1_A\otimes 1_A]\oplus (\bar A\otimes \Bbbk 1_A)
\oplus (\Bbbk 1_A\otimes \bar A))\otimes 
(\Bbbk[1_B\otimes 1_B]\oplus (\bar B\otimes \Bbbk 1_B)
\oplus (\Bbbk 1_B\otimes \bar B))
\end{align*}
and
\[M=[(\bar A\otimes \Bbbk 1_A)\otimes (\Bbbk 1_B\otimes \bar B)]
\oplus [(\Bbbk 1_A\otimes \bar A)\otimes (\bar B\otimes \Bbbk 1_B)] \]
is a sub-bimodule of $E$. By easy computation, we have 
\begin{align*}
E/M\cong & \Bbbk[1_{A\otimes B}\otimes 1_{A\otimes B}]
\oplus [(\Bbbk 1_A\otimes \bar B)\otimes \Bbbk 1_{A\otimes B}]
\oplus [(\Bbbk 1_{A\otimes B}\otimes (\Bbbk 1_A\otimes \bar B))]\\
& \oplus [(\bar A\otimes \Bbbk 1_B)\otimes \Bbbk 1_{A\otimes B} ]
\oplus [(\bar A\otimes \bar B) \otimes \Bbbk 1_{A\otimes B}]
\oplus [\Bbbk 1_{A\otimes B}\otimes (\bar A\otimes \Bbbk 1_B)]
\oplus [\Bbbk 1_{A\otimes B}\otimes (\bar A\otimes \bar B)]\\
=& \Bbbk[1_{A\otimes B}\otimes 1_{A\otimes B}]
\oplus (\overline{A\otimes B}\otimes \Bbbk 1_{A\otimes B})
\oplus (\Bbbk 1_{A\otimes B}\otimes \overline{A\otimes B})\\
=& [(A \otimes B)\otimes (A\otimes B)]/[(\overline{A \otimes B})
\otimes (\overline{A \otimes B})]
\end{align*}
Therefore, we obtain a trident $(A\otimes B, M, E)$. Finally this trident 
$(A\otimes B,M,E)$ is denoted by $A\odot B$. Using the language 
of trident system, we have 

\begin{enumerate}
\item[(1)]
$f\colon \overline{A\otimes B}\to M$ is determined by
$$\begin{aligned}
f(a\otimes 1_B)&=0,\\
f(1_A\otimes b)&=0,\\
f(a\otimes b) &=(a\otimes 1_A)\otimes (1_B\otimes b)+
(1_A\otimes a)\otimes (b\otimes 1_B)
\end{aligned}
$$
for all $a\in {\bar A}$ and $b\in {\bar B}$.
\item[(2)]
$g\colon \overline{A\otimes B}\otimes \overline{A\otimes B}\to M$ 
is determined by
$$\begin{aligned}
g((a\otimes 1_B)\otimes (a'\otimes 1_B))&=0,\\
g((a\otimes 1_B)\otimes (1_A\otimes b'))&
   =(a\otimes 1_A)\otimes (1_B\otimes b'),\\
g((a\otimes 1_B)\otimes (a'\otimes b'))&=0,\\
g((1_A\otimes b)\otimes (a'\otimes 1_B))&
   =(1_A\otimes a')\otimes (b\otimes 1_B),\\
g((1_A\otimes b)\otimes (1_A\otimes b'))&=0,\\
g((1_A\otimes b)\otimes (a'\otimes b'))&=0,\\
g((a\otimes b)\otimes (a'\otimes 1_B))&=0,\\
g((a\otimes b)\otimes (1_A\otimes b'))&=0,\\
g((a\otimes b)\otimes (a'\otimes b'))&=0
\end{aligned}
$$
for all $a, a'\in {\bar A}$ and $b,b'\in {\bar B}$.
\end{enumerate}
\end{example}

\section{Classification of 2-unitary operads of GK-dimension 3}
\label{xxsec4}
\subsection{An operad constructed from a trident algebra}
\label{xxsec4.1}
In this part, we construct a 2-unitary operad $\ip$ by $(A, M, E)$, 
with $\ip(0)=\Bbbk \1_0$, $\ip(1)=A$ and $\ip(2)=E$,
where the composition $\ip(1)\ucr{1} \ip(1)\to \ip(1)$ is given 
by the multiplication of $A$, the compositions 
$\ip(1) \ucr{1} \ip(2)\to \ip(2)$,
$\ip(2)\circ (\ip(1), \ip(1)) \to \ip(1)$ are given by the 
corresponding actions of $A$ on $E$. 

In fact, let $(A, M, f, g)$ be a trident system. We consider 
the operad $\ip$ generated by the $\Bbbk\S$-module 
$(\Bbbk\1_0, A, \Bbbk\1_2\oplus M, 0, 0, \cdots)$
and subject to the following relations
\begin{align*}
&a\circ \1_0=0,\ {\rm for\ all}\ a\in \bar A, \\
&\mu\underset{i}\circ \1_0=0, \ {\rm for\ all}\  \mu\in M, \\
&\1_2\underset{i}\circ\1_0=\1_1, \ {\rm for}\ i=1, 2, \\
&a\circ b=ab, \ {\rm for\ all}\ a, b\in \bar A, \\
&a\circ\mu=a\cdot \mu, \ {\rm for\ all}\ a\in \bar A, \mu\in M, \\
&a\circ \1_2=\1_2\underset{1}\circ a
 +\1_2\underset{2}\circ a+f(a), \ {\rm for\ all}\ a\in \bar A, \\
&\1_2\circ (a, b)=g(a, b), \ {\rm for\ all}\ a, b\in \bar A,\\
&\mu\circ (a, b)=\mu\cdot(a\otimes b), \ {\rm for\ all}\ \mu\in M, a, b\in \bar A, \\
&\1_2\underset{1}\circ\1_2 =\1_2\underset{2}{\circ}\1_2,\\
& \mu\underset{1}\circ \1_2=\1_2\underset{2}
 \circ\mu+(\1_2\underset{2}\circ\mu)\ast (2, 1, 3), \ {\rm for\ all}\ \mu\in M, \\
& \mu\underset{2}\circ \1_2=\1_2\underset{1}\circ\mu
 +(\1_2\underset{1}\circ\mu)\ast (1, 3, 2), \ {\rm for\ all}\ \mu\in M,\\
& \mu\underset{i}\circ \mu'=0,  \ {\rm for\ all}\ \mu, \mu'\in M, i=1, 2.
\end{align*}
where $\1_2\ast (2,1)=\1_2$, and $ab, a\cdot\mu, \mu\cdot(a\otimes b)$ are
given by the multiplication of $A$, the left module action of $A$ on $M$,
and the right module action of $A\otimes A$ on $M$, respectively.

Next we give an explicit description of $\ip$.

\noindent
(C1) The vector space $\ip(n)$:

\begin{itemize}
\item[(C11)] $\ip(0)=\Bbbk \1_0$.
\item[(C12)] $\ip(1)=A=\Bbbk \1_1 \oplus \bar A$.
\item[(C13)] $\ip(2)=\Bbbk \1_2\oplus (\bar A^{(2)}_1\oplus
       \bar A^{(2)}_2) \oplus M$.
\item[(C14)] for each $n\ge 3$,
\[\ip(n)=\Bbbk \1_n\oplus \bigoplus\limits_{k=1}^n \bar A^{(n)}_k
        \oplus \bigoplus\limits_{1\le i<j\le n} M^{(n)}_{ij},\]
where $\Bbbk\1_n$ is a 1-dimensional vector space with the basis
$\1_n$, $\bar A^{(n)}_k$ is a vector space
isomorphic to $\bar A$ for $1\le k\le n$, $n\ge 2$, and
$M^{(n)}_{ij}$ is a vector space isomorphic to $M$
for $1\le i<j\le n$, $n\ge 2$.
\end{itemize}

In order to write elements in $\bar A^{(n)}_k$ and $M^{(n)}_{ij}$,
we fix two families of $\Bbbk$-linear isomorphisms
\begin{center}
$\varphi^n_k\colon \bar A \to \bar A^{(n)}_k \quad$ and
$\quad \psi^n_{ij}\colon M \to M^{(n)}_{ij}$,
\end{center}
for $1\le k\le n, 1\le i<j\le n$ and $n\ge 2$. In fact, we
will see that $\bar A^{(n)}_k=\{\1_n\underset{k}{\circ}a\mid a\in \bar A\}$
and $M^{(n)}_{ij}=\{(\1_{n-1}\underset{1}{\circ} \mu)
\ast c_{ij}\mid \mu\in M \}$, where  $c_{ij}=(i, j, 1, \cdots, i-1,
\hat{i}, i+1, \cdots, j-1, \hat{j}, j+1, \cdots, n)$.

\noindent
(C2) The right action of $\Bbbk \S_n$ on $\ip(n)$: for each $\sigma\in \S_n$,

\begin{itemize}
\item[(C21)]
$\1_n\ast \sigma = \1_n$,
\item[(C22)]
$\varphi^{(n)}_i(a) \ast \sigma =  \varphi^{(n)}_{\sigma^{-1}(i)}(a)$,
\item[(C23)]
$\psi^{(n)}_{ij}(\mu)\ast \sigma=
\begin{cases}
\psi^{(n)}_{\sigma^{-1}(i), \sigma^{-1}(j)}(\mu),
  & \sigma^{-1}(i)<\sigma^{-1}(j), \\
\psi^{(n)}_{\sigma^{-1}(j), \sigma^{-1}(i)}(\mu\ast (21)),
  & \sigma^{-1}(i)>\sigma^{-1}(j).
\end{cases}$
\end{itemize}

\noindent
(C3) The partial composition $\ip(m) \underset{s}{\circ} \ip(n) \to \ip(m+n-1)$:

\begin{itemize}
\item[(C31)]
$\1_m \underset{s}\circ \1_n=\1_{m+n-1}$.
\item[(C32)]
$\1_m \underset{s}\circ \varphi^{(n)}_{i}(a)=\varphi^{(m+n-1)}_{s+i-1}(a)$.
\item[(C33)]
$\1_m \underset{s}\circ \psi^{(n)}_{i_1,i_2}(\mu )
=\psi^{(m+n-1)}_{s+i_1-1,s+i_2-1}(\mu)$.
\item[(C34)]
$\varphi^{(m)}_i(a) \underset{s}{\circ} \1_n=
\begin{cases}
\varphi^{(m+n-1)}_i(a), & i<s,\\
\sum\limits_{k=i}^{i+n-1}\varphi^{(m+n-1)}_{k}(a)
+\sum\limits_{i\le k_1<k_2\le i+n-1}\psi^{(m+n-1)}_{k_1k_2}(f(a)), & i=s, \\
\varphi^{(m+n-1)}_{i+n-1}(a), & i>s.
\end{cases}$
\item[(C35)]
$\varphi^{(m)}_i(a) \underset{s}{\circ} \varphi^{(n)}_j(b)=
\begin{cases}
\psi^{(m+n-1)}_{i, s+j-1}(g(a, b)), & i<s, \\
\begin{array}{ll}
\sum\limits_{k=i}^{i+j-2} \psi^{(m+n-1)}_{k, i+j-1}(g(a, b)
+f(a)\underset{2}\cdot b)+\varphi^{(m+n-1)}_{i+j-1}(ab)\\
+\sum\limits_{k=i+j}^{i+n-1} \psi^{(m+n-1)}_{i+j-1, k}(g(b, a)
+f(a)\underset{1}\cdot b),
\end{array}& i=s, \\
\psi^{(m+n-1)}_{s+j-1, i+n-1}(g(b, a)), & i>s.
\end{cases}$
\item[(C36)]
$\varphi^{(m)}_i(a)\underset{s}{\circ} \psi^{(n)}_{j_1j_2}(\mu)=
\begin{cases}
0, & i\neq s, \\
\psi^{(m+n-1)}_{i+j_1-1, i+j_2-1}(a\mu), & i=s.\\
\end{cases}$
\item[(C37)]
$\psi^{(m)}_{i_1i_2}(\mu) \underset{s}{\circ} \1_n=
\begin{cases}
\psi^{(m+n-1)}_{i_1+n-1, i_2+n-1}(\mu), & 1\le s<i_1, \\
\sum\limits_{k=i_1}^{i_1+n-1} \psi^{(m+n-1)}_{k, i_2+n-1}(\mu), & s=i_1, \\
\psi^{(m+n-1)}_{i_1, i_2+n-1}(\mu), & i_1<s<i_2,\\
\sum\limits_{k=i_2}^{i_2+n-1} \psi^{(m+n-1)}_{i_1, k}(\mu), & s=i_2,\\
\psi^{(m+n-1)}_{i_1i_2}(\mu), & i_2<s\le m.
\end{cases}$
\item[(C38)]
$\psi^{(m)}_{i_1i_2}(\mu) \underset{s}{\circ} \varphi^{(n)}_j(b)=
\begin{cases}
0, & s\neq i_1, i_2, \\
\psi^{(m+n-1)}_{i_1+j-1, i_2+n-1}(\mu \underset{1}\cdot b), & s=i_1, \\
\psi^{(m+n-1)}_{i_1, i_2+j-1}(\mu \underset{2}\cdot b), & s=i_2.\\
\end{cases}$
\item[(C39)]
$\psi^{(m)}_{i_1i_2}(\mu) \underset{s}{\circ} \psi^{(n)}_{j_1j_2}(\nu) =0.$
\end{itemize}

\begin{theorem}
\label{xxthm4.1}
Retain the above notation. Let $(A,M,f,g)$ be a trident algebra. Then
$\ip:=F(A,M,f,g)$ is a 2-unitary $\Com$-augmented operad of
GK-dimension 3.
\end{theorem}

A tedious proof of Theorem \ref{xxthm4.1} is given in the final section.

\subsection{Classification of 2-unitary operads of GK-dimension 3}
\label{xxsec4.2}
Now we are ready to prove the main theorem.

\begin{theorem}
\label{xxthm4.2}
The category $\C$ consisting of trident algebras $(A, M, f, g)$ is
equivalent to the category $\mathcal{D}$ of $\Com$-augmented
operads of GK-dimension 3.
\end{theorem}

\begin{proof}
Define a functor ${\mathcal{F}}: \mathcal{C} \longrightarrow
\mathcal{D}$ as follows:
\begin{enumerate}
\item[(i)]
For any trident algebra $(A, M, f, g)$,
$${\mathcal{F}}(A, M, f, g):=F(A,M,f,g)$$
where $F(A,M,f,g)$ is given in Theorem \ref{xxthm4.1}, namely,
$${\mathcal{F}}(A, M, f, g)
=\{\ip(n)\}_{n\geq 0}=\{\Bbbk \1_n\oplus \bigoplus\limits_{k=1}^n
\bar A^{(n)}_k \oplus 
\bigoplus\limits_{1\le i<j\le n} M^{(n)}_{ij}\}_{n\geq 0}.$$
By Theorem \ref{xxthm4.1}, ${\mathcal{F}}(A, M, f, g)$ is a
$\Com$-augmented operad of GK-dimension 3.
\item[(ii)]
For a morphism
$(\alpha, \beta)\colon (A, M, f, g)\to (A', M', f', g')$, we
define an operadic morphism
$$\Phi={\mathcal{F}}(\alpha, \beta)\colon\{\ip(n)\} \to \{\ip'(n)\}$$
as follows:
\begin{align*}
\Phi_n(\1_n)\colon
    &=\1'_n;\\
\Phi_n(\varphi^{n}_k(a))\colon
    &={\varphi' }^n_k(\alpha(a)), \ \ {\rm for}\ a\in \bar{A}; \\
\Phi_n(\psi^n_{ij}(\mu))\colon
    &={\psi'}^n_{ij} (\beta(\mu)), \ \ {\rm for}\ \mu\in M.
\end{align*}
\end{enumerate}

By a direct calculation, it follows easily from (C21)-(C23)
and (C31)-(C39) that $\Phi$ is a morphism of operads since
$(\alpha, \beta)$ is a morphism in the category $\C$.

Conversely, we define a functor ${\mathcal G}: \mathcal{D}
\longrightarrow \mathcal{C}$ as follows: for an object $\ip$
in category $\mathcal{D}$, we define
$${\mathcal G}(\ip)=(A,M,f,g)$$
where $A=\ip(1)$, $M=\up{2}_\ip(2)$, and
\begin{align*}
f: \quad & \up{1}(1) \rightarrow \up{2}(2),
& a\mapsto a\circ\1_2-\1_2\underset{1}{\circ}a
  -\1_2\underset{2}{\circ}a; \\
g: \quad & \up{1}(1) \ot \up{1}(1) \rightarrow \up{2}(2),
& (a, b)\mapsto (\1_2\underset{1}{\circ}a)\underset{2}{\circ}b.
\end{align*}

We show next that $(A,M,f,g)$ is a trident algebra. By
definitions, $A:=\ip(1)$ is an associative $\Bbbk$-algebra
with identity $\1_1$. Considering the map
$\pi^{\varnothing}: \ip(1)\rightarrow \ip(0)=\Bbbk \1_0,
\theta\mapsto \theta\circ \1_0$, we know that $A$ is an
augmented algebra with the augmentation ideal
$\Ker \pi^{\varnothing}=\up{1}_\ip(1)$. By the definition of
truncation ideals, $M:=\up{2}_\ip(2)$ is a $\Bbbk\S_2$-submodule
of $\ip(2)$, and  is an $(A, A\ot A)$-bimodule with the module
actions given by the related composition map. Since $\ip$ is a
$\Com$-augmented (hence 2-unitary) operad of GK-dimension 3, we
have
\begin{equation}
\notag 
(a\circ\1_2-\1_2\underset{1}{\circ}a-\1_2\underset{2}{\circ}a)\ucr{i} \1_0 =0,
\end{equation}
for all $a\in {\bar A}$ and $i=1,2$, and
\begin{equation}
\notag 
( (\1_2\underset{1}{\circ}a)\underset{2}{\circ}b)\ucr{i} \1_0 =0,
\end{equation}
for all $a,b\in {\bar A}$ and $i=1,2$. Therefore
$a\circ\1_2-\1_2\underset{1}{\circ}a-\1_2
\underset{2}{\circ}a$ and
$(\1_2\underset{1}{\circ}a)\underset{2}{\circ}b$ are
in $M$. Therefore $f$ maps from ${\bar A}\to M$ and
$g$ maps from ${\bar A}^{\otimes 2}\to M$.

For $a,b\in A$ and $\mu\in M$, let $a\cdot \mu$ be $a\circ \mu$
and $\mu\cdot (a\otimes b)=\mu\circ(a,b)$ both of which are in
$M$. Then, for all $a\in A$, $\mu\in M$,
$$
a\cdot (\mu\ast (2,1))= a\underset{1}{\circ}(\mu\ast (2,1))
= (a\underset{1}{\circ}\mu)\ast(\1_1\underset{1}{\circ} (2,1))
= (a\cdot \mu)\ast (2,1)
$$
which shows that \eqref{E3.1.1} holds. For $a,b\in A$, we have
\begin{align*}
(\mu\ast (2,1))\cdot (a\otimes b)=&(\mu\ast (2,1))\underset{1}{\circ}a\underset{2}{\circ}b
=((\mu\underset{2}{\circ}a)\ast ((2,1)\underset{1}{\circ}\1_1))\underset{2}{\circ}b\\
=&((\mu\underset{2}{\circ}a)\underset{1}{\circ}b)\ast ((2,1)\underset{2}{\circ}\1_1))
= (\mu\cdot(b\otimes a))\ast (2,1),
\end{align*}
which shows that \eqref{E3.1.2} holds. Hence $M$ is a trident $A$-module.

For $\bar{a},\bar{b}\in {\bar A}$,
\begin{align*}
f(\bar{a}\bar{b})
=&(\bar{a}\bar{b})\circ\1_2-\1_2\underset{1}{\circ}
 (\bar{a}\bar{b})-\1_2\underset{2}{\circ}(\bar{a}\bar{b})\\
=&\bar{a}\circ (\bar{b}\circ\1_2-\1_2\underset{1}{\circ}\bar{b}
 -\1_2\underset{2}{\circ}\bar{b})
 +(\bar{a}\circ\1_2-\1_2\underset{1}{\circ}\bar{a}
 -\1_2\underset{2}{\circ}\bar{a})\underset{1}{\circ}\bar{b}\\
&+(\bar{a}\circ\1_2-\1_2\underset{1}{\circ}\bar{a}
 -\1_2\underset{2}{\circ}\bar{a})\underset{2}{\circ}\bar{b}
 +(\1_2\underset{1}{\circ}\bar{a})\underset{2}{\circ}\bar{b}
 +(\1_2\underset{1}{\circ}\bar{b})\underset{2}{\circ}\bar{a}\\
=&\bar{a}f(\bar{b})+f(\bar{a})\underset{1}{\circ}\bar{b}
+f(\bar{a})\underset{2}{\circ}\bar{b}+g(\bar{a}, \bar{b})
+g(\bar{b}, \bar{a}).
\end{align*}
Hence \eqref{E3.4.1} holds.
For $\bar{a},\bar{b}, \bar{c}\in {\bar A}$,
\begin{align*}
f(\bar{a})\cdot(\bar{b},\bar{c})
=&((\bar{a}\circ\1_2-\1_2\underset{1}{\circ}\bar{a}
 -\1_2\underset{2}{\circ}\bar{a})\underset{1}{\circ}\bar{b})
 \underset{2}{\circ}\bar{c}\\
=&\bar{a}((\1_2\underset{1}{\circ}\bar{b})\underset{2}{\circ}\bar{c})
-(\1_2\underset{1}{\circ} (\bar{a}\bar{b}))\underset{2}{\circ}\bar{c}
-(\1_2\underset{1}{\circ}\bar{b})\underset{2}{\circ}(\bar{a}\bar{c})\\
=&\bar{a}g(\bar{b}, \bar{c})-g(\bar{a}\bar{b}, \bar{c})
-g(\bar{b}, \bar{a}\bar{c}).
\end{align*}
Hence \eqref{E3.4.2} holds. Therefore $(A, M, f, g)$ is a
trident algebra. It follows that ${\mathcal G}(\ip)$ is an 
object in the category $\C$.

For the operad morphism $\Psi \colon \ip \to \ip'$, we define
the morphism
\[(\alpha, \beta)={\mathcal G}(\Psi)\colon (\ip(1),\up{2}_\ip(2),f,g)
\to (\ip'(1),\up{2}'_{\ip'}(2),f',g')\]
as follows:
\[{\mathcal G}(\Psi): =(\Psi(1),\Psi(2)|_{\up{2}(2)}).\]
Therefore ${\mathcal G}$ is a functor from ${\mathcal D}\to \C$.

Finally, it is clear from the definition that ${\mathcal G}
{\mathcal F}$ is the identity and it follows from Lemma \ref{xxlem2.4}
that ${\mathcal F}{\mathcal G}$ is naturally isomorphic to the identity.
The assertion follows.
\end{proof}

\section{Comments, examples, and remarks}
\label{xxsec5}

We refer to \cite{LV} for the definition of Hadamard product
$-\underset{{\rm H}}{\otimes}-$ and a Hopf operad. Recall that 
the \textit{Hadamard product} $\ip\underset{\textrm{H}}{\otimes} \iq$ 
of the operads $\ip$ and $\iq$ is defined to be 
\begin{align*}
(\ip\underset{\textrm{H}}{\otimes} \iq)(n)= \ip(n)\otimes \iq(n),
\end{align*}
for all $n\ge 0$ with the partial composition 
\begin{align*}
 (\mu_1\otimes \nu_1)\ucr{i} (\mu_2\otimes \nu_2)
=(\mu_1\ucr{i}\mu_2)\otimes (\nu_1\ucr{i}\nu_2),
\end{align*} 
for $\mu_1\otimes \nu_1\in (\ip\underset{\textrm{H}}{\otimes} \iq)(m)$,
$\mu_2\otimes \nu_2\in (\ip\underset{\textrm{H}}{\otimes} \iq)(n)$, 
and $m\ge 1, n\ge 0, 1\le i\le m$. Clearly, the operad $\Com$ is 
obviously a unit for Hadamard product.

A \textit{Hopf operad} is a symmetric operad $\ip$ with a morphism 
of operads $\Delta \colon \ip \to \ip\underset{\textrm{H}}\otimes \ip$
called the coproduct of $\ip$ and a morphism of operads
$\epsilon_\ip\colon \ip \to \Com$ called the counit,
which is supposed to be coassociative and counital. 

\begin{definition}
\label{xxdef5.1}
Let $\ip$ be a Hopf operad. We say that $\ip$ is a
{\it $\Com$-augmented Hopf operad} if
\begin{enumerate}
\item[(1)]
$\ip$ is $\Com$-augmented and the composition
$$\Com \xrightarrow{\; u_{\ip}\; } \ip \xrightarrow{\; \epsilon \; } \Com$$
is the identity map, and
\item[(2)]
the following diagram is commutative
$$\begin{CD}
\Com @> \cong >> \Com \hot \Com\\
@V u_{\ip} VV @VV u_{\ip}\hot u_{\ip} V\\
\ip @>> \Delta> \ip\hot \ip.
\end{CD}$$
\end{enumerate}
\end{definition}

\begin{remark}
\label{xxrem5.2}
A Hopf operad satisfying the condition (1) in Definition 
\ref{xxdef5.1}  is also called a unital augmented connected 
Hopf operad in \cite[Definition 2.5]{Kh}.
\end{remark}

\begin{proposition}
\label{xxpro5.3}
Let $\ip$ be a $\Com$-augmented
Hopf operad of GK-dimension $\leq 2$. Then $\ip=\Com$.
\end{proposition}

\begin{proof} Note that $\Delta: \ip \to \ip\hot \ip$ is an 
operadic morphism. Write $\ip$ as $F(A,0,0,0)$ as given by 
Theorem \ref{xxthm0.2}. Suppose ${\bar A}\neq 0$. Then
$\gkdim \ip=2$ and $\gkdim \ip\hot \ip=3$. Since
$\ip\hot \ip$ has GK-dimension 3, it is of the form
$F(A\odot A)$ where $A\odot A$ is given in Example
\ref{xxex3.8}. By Theorem \ref{xxthm4.2},
$\Delta$ induces a morphism of trident algebras
$(A,0,0,0)\to A\odot A:=(A\otimes A, M,f,g)$. Then
$f \Delta\mid_{\ip(1)}: A\to A\otimes A \to M$ is zero. 
We claim that $A=\Bbbk$. If not, let
$0\neq a\in {\bar A}$ and write $\Delta\mid_{\ip(1)}(a)=
1\otimes a+ a\otimes 1+ \sum a_{(1)}\otimes a_{(2)}$ where
$a_{(1)}, a_{(2)}\in {\bar A}$. Then, by three equations in 
Example \ref{xxex3.8}(1),
$$\begin{aligned}
0&=f\Delta(a)=f( 1\otimes a+ a\otimes 1+ \sum a_{(1)}\otimes
a_{(2)})\\
&=\sum ((a_{(1)}\otimes 1_A) \otimes (1_A\otimes a_{(2)})
+(1_A \otimes a_{(1)}) \otimes (a_{(2)}\otimes 1_A)).
\end{aligned}
$$
Therefore $\sum a_{(1)}\otimes a_{(2)}=0$, and consequently,
$a$ is a primitive element. By Definition \ref{xxdef5.1}(2),
$\1_n$ is group-like, i.e., $\Delta(\1_n)=\1_n \otimes \1_n$ 
for all $n$. Since each $a$ is primitive, it follows from 
(C32) that each $\varphi_i^{(n)}(a)$ is primitive,
i.e., $\Delta(\varphi_i^{(n)}(a))=
\varphi_i^{(n)}(a)\otimes \1_n+\1_n\otimes \varphi_i^{(n)}(a)$
for all $i,n$. Since $\ip$ has GK-dimension $\leq 2$, 
$\varphi_1^{(2)}(a) \ucr{1} \varphi_1^{(3)}(a)=0$ by (C35). But
$$\begin{aligned}
\Delta(\varphi_1^{(2)}(a) \ucr{1} \varphi_1^{(3)}(a))
&=
(\varphi_1^{(2)}(a)\otimes \1_2+\1_2\otimes \varphi_1^{(2)}(a))
\ucr{1}
(\varphi_1^{(3)}(a)\otimes \1_3+\1_3\otimes \varphi_1^{(3)}(a))\\
&=(\varphi_1^{(2)}(a)\otimes \1_2)\ucr{1}
(\1_3\otimes \varphi_1^{(3)}(a))+
(\1_2\otimes \varphi_1^{(2)}(a))\ucr{1}
(\varphi_1^{(3)}(a)\otimes \1_3)\\
&=(\sum_{k=1}^{2} \varphi_k^{(4)}(a))\otimes \varphi_1^{(4)}(a)+
\varphi_1^{(4)}(a)\otimes (\sum_{k=1}^{2} \varphi_k^{(4)}(a))\\
&=\varphi_2^{(4)}(a)\otimes \varphi_1^{(4)}(a)+
\varphi_1^{(4)}(a)\otimes \varphi_2^{(4)}(a)+ 2 \varphi_1^{(4)}(a)\otimes \varphi_1^{(4)}(a)\\
&\neq 0,
\end{aligned}
$$
yielding a contradiction. Therefore ${\bar A}=0$ and $\ip=\Com$.
%
%
%
\end{proof}


Unlike in the GK-dimension 2 case, the 2-unit of a 2-unitary
operad needs not be unique.

\begin{example}
\label{xxex5.4}
Let $\ip=F(A,M,f,g)$ where $M\neq 0$. Let $\1_2$ be the
canonical 2-unit of $\ip$ given in the construction of
$F(A,M,f,g)$. Let $\1'_2=\1_2+ \psi^{(2)}_{12}(m)$. It is easy to
check that $\1'_2$ is a $2a$-unit. Suppose that $m\ast (2,1)\neq
m$. Then $(\ip, \1_0, \1_1,\1'_2)$ is a 2-unitary
operad, but not $\Com$-augmented.
As a consequence, we can not replace ``$\Com$-augmented'' by
``2-unitary'' in Theorem \ref{xxthm0.2}.
\end{example}

\begin{remark}
\label{xxrem5.5}
For non-2-unitary operad, we have the following remarks.
\begin{enumerate}
\item[(1)]
By \cite[Construction 7.1]{QXZZ} there are a lot of symmetric
operads of GK-dimension 3 that are not 2-unitary.
\item[(2)]
In \cite{QXZZ}, an analogue of Bergman's gap theorem of
nonsymmetric operads is proved, namely, no finitely generated
locally finite nonsymmetric operad has GK-dimension strictly
between 1 and 2. In \cite{LQXZZZ} the authors proved that there
is no finitely generated symmetric operad with GK-dimension
strictly between 1 and 2.
\item[(3)]
It is an open question if there are finitely generated symmetric
operads with GK-dimension strictly between 2 and 3, see
\cite[Question 0.8]{QXZZ}.
\item[(4)]
For every $r\in \{0\}\cup \{1\} \cup [2, \infty)$ or $r=\infty$,
the authors in \cite{QXZZ} constructed an explicit non-symmetric
operad of GK-dimension $r$.
\end{enumerate}
\end{remark}

The following lemma was proved in \cite[Theorem 6.5]{BYZ}.

\begin{lemma}
\label{xxlem5.6}
Let $\ip$ be a 2-unitary operad of finite GK-dimension $\geq 3$.
Then $\ip$ is not semiprime.
\end{lemma}

\begin{proof}
If $\ip$ is semiprime, by the proof of \cite[Theorem 6.5]{BYZ},
$\up{(2)}=0$. So $\gkdim \ip\leq 2$, yielding a contradiction.
\end{proof}

The following example shows that $\hot$ does not
preserve primeness.

\begin{example}
\label{xxex5.7}
Let $A=\Bbbk\oplus M_2(\Bbbk)$. Then $\ip:=F(A,0,0,0)$ is prime
of GK-dimension 2 by Theorem \ref{xxthm0.1}. Since
$\ip\hot \ip$ has GK-dimension 3, it is not semiprime
by Lemma \ref{xxlem5.6}.
\end{example}

\section{Appendix: Proof of Theorem \ref{xxthm4.1}}
\label{xxsec6}
This final section is devoted to a complete proof of 
Theorem \ref{xxthm4.1}.

\begin{proof}[Proof of Theorem \ref{xxthm4.1}]
We need to check (OP1), (OP2), (OP3).
Let ${\mathbb A}=\{\1_{n}\}$,
$\Phi=\{\varphi_{i}^{(n)}(a) \mid a\in {\bar A}\}$, and
$\Psi=\{\psi_{ij}^{(n)}(y) \mid y\in M\}$. We will use these
elements.

\noindent
{\it Verification of (OP1):}
By (C31), (C32) and (C33), $\1 \underset{1}{\circ} \theta=\theta$
for all $\theta\in \ip$. By (C31), (C34) and (C37), we
have $\theta\underset{i}{\circ} \1=\theta$ for all
$\theta\in \ip$ and $1\leq i\leq \Ar(\theta)$. Therefore
(OP1) holds.

\noindent
{\it Verification of (OP2):} There are two equations in (OP2).
We only check the first one in (OP2), namely, the following
equation
\begin{equation}
\label{E1.1.1a}\tag{E1.1.1}
(\la  \underset{i}{\circ} \mu) \underset{i-1+j}{\circ} \nu
=\la \underset{i}{\circ} (\mu \underset{j}{\circ} \nu),
1\le i\le l, 1\le j\le m
\end{equation}
for all $\la\in \ip(l), \mu\in \ip(m)$ and
$\nu\in \ip(n)$.

If two of
$\lambda, \mu,$ and $\nu$ are $\psi^{(n)}_{ij}$, it follows
from (C39) that both sides of \eqref{E1.1.1} are zero. Then
there are following 20 cases to consider.

{\sf Case 1: $\lambda \in \Psi$ and $\mu,\nu\in \Phi$. }
Write $\lambda =\psi^{(l)}_{k_1k_2}(y)$, $\mu=\varphi^{(m)}_{s}(a)$,
and $\nu=\varphi^{(n)}_{t}(b)$. Then
$$\begin{aligned}
{\text{LHS of \eqref{E1.1.1}}}
&=(\psi^{(l)}_{k_1k_2}(y)\underset{i}{\circ} \varphi^{(m)}_{s}(a))
\underset{i-1+j}{\circ} \varphi^{(n)}_{t}(b)\\
&\underset{\text{by (C38)}}{=}{\tiny
\begin{cases}
0 & i\neq k_1,k_2,\\
\psi_{k_1+s-1,k_2+m-1}^{(l+m-1)}(y\underset{1}{\cdot}a) & i=k_1,\\
\psi_{k_1,k_2+s-1}^{(l+m-1)}(y\underset{2}{\cdot}a) & i=k_2,
\end{cases}} \qquad \underset{i-1+j}{\circ} \varphi^{(n)}_{t}(b) \\
&\underset{\text{by (C38)}}{=}{\tiny
\begin{cases}
0 & i\neq k_1, k_2, i+j-1\neq k_1+s-1, k_2+m-1, k_1, k_2+s-1,\\
\psi_{k_1+s+t-2,k_2+m+n-2}^{(l+m+n-2)}((y\underset{1}{\cdot}a)\underset{1}{\cdot}b)
& i=k_1, i+j-1=k_1+s-1,\\
\psi_{k_1+s-1,k_2+m+t-2}^{(l+m+n-2)}((y\underset{1}{\cdot}a)\underset{2}{\cdot}b)
& i=k_1, i+j-1=k_2+m-1, \; (impossible)\\
\psi_{k_1+t-1,k_2+s+n-2}^{(l+m+n-2)}((y\underset{2}{\cdot}a)\underset{1}{\cdot}b)
& i=k_2, i+j-1=k_1,\; (impossible)\\
\psi_{k_1,k_2+s+t-2}^{(l+m+n-2)}((y\underset{2}{\cdot}a)\underset{2}{\cdot}b)
& i=k_2, i+j-1=k_2+s-1,
\end{cases}}\\
&=
{\tiny
\begin{cases}
\psi_{k_1+s+t-2,k_2+m+n-2}^{(l+m+n-2)}((y\underset{1}{\cdot}a)\underset{1}{\cdot}b)
& i=k_1, j=s,\\
\psi_{k_1,k_2+s+t-2}^{(l+m+n-2)}((y\underset{2}{\cdot}a)\underset{2}{\cdot}b)
& i=k_2, j=s,\\
0& {\text{otherwise}},
\end{cases}}
\end{aligned}
$$
and
$$\begin{aligned}
{\text{RHS of \eqref{E1.1.1}}}
&=\psi^{(l)}_{k_1k_2}(y)\underset{i}{\circ} (\varphi^{(m)}_{s}(a)
\underset{j}{\circ} \varphi^{(n)}_{t}(b))\\
&\underset{\text{by (C35)}}{=}
\psi^{(l)}_{k_1k_2}(y)\underset{i}{\circ}
{\tiny \begin{cases}
\psi^{(m+n-1)}_{s, j+t-1}(g(a, b)), & s<j, \\
\begin{array}{ll}
\sum\limits_{k=s}^{s+t-2} \psi^{(m+n-1)}_{k, s+t-1}(g(a, b)
+f(a)\underset{2}\cdot b)+\varphi^{(m+n-1)}_{s+t-1}(ab)\\
+\sum\limits_{k=s+t}^{s+n-1} \psi^{(m+n-1)}_{s+t-1, k}(g(b, a)
+f(a)\underset{1}\cdot b),
\end{array}& s=j, \\
\psi^{(m+n-1)}_{j+t-1, s+n-1}(g(b, a)), & s>j.
\end{cases}}\\
&\underset{\text{by (C39)}}{=}{\tiny
\begin{cases}
0 & s<j,\\
\psi^{(l)}_{k_1k_2}(y)\underset{i}{\circ}
\varphi^{(m+n-1)}_{s+t-1}(ab) & s=j,\\
0 & s>j,
\end{cases}}\\
&\underset{\text{by (C36)}}{=}
{\tiny
\begin{cases}
0 & s<j,\\
0& s=j, i\neq k_1, k_2,\\
\psi^{(l+m+n-2)}_{k_1+s+t-2,k_2+m+n-2}(y\underset{1}{\cdot}ab)
& s=j, i=k_1\\
\psi^{(l+m+n-2)}_{k_1,k_2+s+t-2}(y\underset{2}{\cdot}ab)
& s=j, i=k_2\\
0 & s>j,
\end{cases}}\\
&=
{\tiny
\begin{cases}
\psi_{k_1+s+t-2,k_2+m+n-2}^{(l+m+n-2)}((y\underset{1}{\cdot}a)\underset{1}{\cdot}b)
& i=k_1, j=s,\\
\psi_{k_1,k_2+s+t-2}^{(l+m+n-2)}((y\underset{2}{\cdot}a)\underset{2}{\cdot}b)
& i=k_2, j=s,\\
0& {\text{otherwise}},
\end{cases}}
\end{aligned}
$$
which implies that \eqref{E1.1.1} holds.

{\sf Case 2: $\mu \in \Psi$ and $\lambda,\nu\in \Phi$. }
Write $\lambda=\varphi^{(l)}_{u}(a)$, $\mu=\psi_{k_1k_2}^{(m)}(y)$,
and $\nu=\varphi^{(n)}_{v}(c)$ where $a,c\in {\bar A}$ and
$b\in M$. Then
$$\begin{aligned}
{\text{LHS of \eqref{E1.1.1}}}
&=(\varphi^{(l)}_{u}(a)\underset{i}{\circ}\psi_{k_1k_2}^{(m)}(y))
\underset{i-1+j}{\circ} \varphi^{(n)}_{v}(c)\\
&\underset{\text{by (C36)}}{=}
{\tiny
\begin{cases}
0& i\neq u,\\
\psi_{k_1+u-1,k_2+u-1}^{(l+m-1)}(ay)& i=u
\end{cases}
} \quad \underset{i-1+j}{\circ} \varphi^{(n)}_{v}(c)\\
&\underset{\text{by (C38)}}{=}
{\tiny
\begin{cases}
0& i\neq u,\\
0&j\neq k_1,k_2,\\
\psi_{k_1+u+v-2,k_2+u+n-2}^{(l+m+n-2)}(ay \underset{1}{\cdot} c) & i=u, j=k_1,\\
\psi_{k_1+u-1,k_2+u+v-2}^{(l+m+n-2)}(ay \underset{2}{\cdot} c) 	& i=u, j=k_2,
\end{cases}
}
\end{aligned}
$$
and
$$\begin{aligned}
{\text{RHS of \eqref{E1.1.1}}}
&=\varphi^{(l)}_{u}(a)\underset{i}{\circ} (\psi_{k_1k_2}^{(m)}(y)
\underset{j}{\circ} \varphi^{(n)}_{v}(c))\\
&\underset{\text{by (C38)}}{=}
\varphi^{(l)}_{u}(a)\underset{i}{\circ}
{\tiny
\begin{cases}
0& j\neq k_1,k_2,\\
\psi_{k_1+v-1,k_2+n-1}^{(m+n-1)}(y\underset{1}{\cdot} c) & j=k_1,\\
\psi_{k_1,k_2+v-1}^{(m+n-1)}(y\underset{2}{\cdot} c) & j=k_2,
\end{cases}
}\\
&\underset{\text{by (C38)}}{=}
{\tiny
\begin{cases}
0& j\neq k_1,k_2,\\
0& i\neq u,\\
\psi_{k_1+u+v-2,k_2+u+n-2}^{(l+m+n-2)}(a(y\underset{1}{\cdot} c))
&j=k_1, i=u,\\
\psi_{k_1+u-1,k_2+u+v-2}^{(l+m+n-2)}(a(y\underset{2}{\cdot} c))
&j=k_2, i=u.
\end{cases}
}
\end{aligned}
$$
Hence \eqref{E1.1.1} holds.

{\sf Case 3: $\nu \in \Psi$ and $\lambda,\mu\in \Phi$.}
Write $\lambda=\varphi^{(l)}_{u}(a)$, $\mu=\varphi^{(m)}_{v}(c)$,
and $\nu=\psi_{k_1k_2}^{(n)}(y)$. Then
$$\begin{aligned}
{\text{LHS of \eqref{E1.1.1}}}
&=(\varphi^{(l)}_{u}(a)\underset{i}{\circ}\varphi^{(m)}_{v}(c))
\underset{i-1+j}{\circ} \psi_{k_1k_2}^{(n)}(y)\\
&\underset{\text{by (C35)}}{=}
{\tiny
\begin{cases}
\psi  &i\neq u,\\
\psi +\varphi_{u+v-1}^{(l+m-1)}(ac) & i=u,
\end{cases}
}
\quad \underset{i-1+j}{\circ} \psi_{k_1k_2}^{(n)}(y)\\
&\underset{\text{by (C39)}}{=}
{\tiny
\begin{cases}
0 &i\neq u,\\
\varphi_{u+v-1}^{(l+m-1)}(ac) \underset{i-1+j}{\circ} \psi_{k_1k_2}^{(n)}(y) & i=u,
\end{cases}
}\\
&\underset{\text{by (C36)}}{=}
{\tiny
\begin{cases}
0 &i\neq u,\\
0 &j\neq v,\\
\psi_{k_1+u+v-2,k_2+u+v-2}^{(l+m+n-2)}((ac)y) & i=u,j=v,
\end{cases}
}
\end{aligned}
$$
where $\psi$ is a linear combination of elements in $\Psi$. And
$$\begin{aligned}
{\text{RHS of \eqref{E1.1.1}}}
&=\varphi^{(l)}_{u}(a)\underset{i}{\circ}(\varphi^{(m)}_{v}(c)
\underset{j}{\circ} \psi_{k_1k_2}^{(n)}(y))\\
&\underset{\text{by (C36)}}{=}
\varphi^{(l)}_{u}(a)\underset{i}{\circ}
{\tiny
\begin{cases}
0& j\neq v,\\
\psi_{k_1+v-1,k_2+v-1}^{(m+n-1)}(cy)& j=v,
\end{cases}
}\\
&\underset{\text{by (C36)}}{=}
{\tiny
\begin{cases}
0& j\neq v,\\
0& i\neq u\\
\psi_{k_1+v+u-2,k_2+v+u-2}^{(l+m+n-2)}(a(cy))& j=v, i=u.
\end{cases}
}
\end{aligned}
$$
Hence \eqref{E1.1.1} holds.

{\sf Case 4: $\lambda \in \Psi$, $\mu\in {\mathbb A}$, and $\nu\in \Phi$. }
Write $\lambda =\psi^{(l)}_{k_1k_2}(y)$, $\mu=\1_m$,
and $\nu=\varphi^{(n)}_{t}(c)$. Then
$$\begin{aligned}
{\text{LHS of \eqref{E1.1.1}}}
&=(\psi^{(l)}_{k_1k_2}(y)\underset{i}{\circ}\1_m)
\underset{i-1+j}{\circ} \varphi^{(n)}_{t}(c)\\
&\underset{\text{by (C37)}}{=}
{\tiny
\begin{cases}
\psi^{(l+m-1)}_{k_1+m-1, k_2+m-1}(y), & 1\le i<k_1, \\
\sum\limits_{k=k_1}^{k_1+m-1} \psi^{(l+m-1)}_{k, k_2+m-1}(y), & i=k_1, \\
\psi^{(l+m-1)}_{k_1, k_2+m-1}(y), & k_1<i<k_2,\\
\sum\limits_{k=k_2}^{k_2+m-1} \psi^{(l+m-1)}_{k_1, k}(y), & i=k_2,\\
\psi^{(l+m-1)}_{k_1k_2}(y), & k_2<i\le l.
\end{cases}
} \quad \underset{i-1+j}{\circ} \varphi^{(n)}_{t}(c)\\
&\underset{\text{by (C38)}}{=}
{\tiny
\begin{cases}
{\tiny
\begin{cases}
0, & i-1+j\neq k_1+m-1,k_2+m-1\\
\psi_{k_1+m+t-2,k_2+m+n-2}^{(l+m+n-2)}(y\underset{1}{\cdot} c), & i-1+j=k_1+m-1,(impossible)\\
\psi_{k_1+m-1,k_2+m+t-2}^{(l+m+n-2)}(y\underset{2}{\cdot} c), & i-1+j=k_1+m-1,(impossible).
\end{cases}
}& 1\le i<k_1,\\
\sum\limits_{k=k_1}^{k_1+m-1}\psi^{(l+m-1)}_{k,k_2+m-1}(y)\underset{i-1+j}{\circ} \varphi^{(n)}_{t}(c)=\psi_{k_1+j+t-2,k_2+m+n-2}^{(l+m+n-2)}(y\underset{1}{\cdot} c), &  i=k_1,\\
{\tiny
\begin{cases}
0, & i-1+j\neq k_1,k_2+m-1\\
\psi_{k_1+t-1,k_2+m+n-2}^{(l+m+n-2)}(y\underset{1}{\cdot} c), & i-1+j=k_1,(impossible)\\
\psi_{k_1,k_2+m+t-2}^{(l+m+n-2)}(y\underset{2}{\cdot} c), & i-1+j=k_2+m-1,(impossible).
\end{cases}
}& k_1<i<k_2,\\
\sum\limits_{k=k_2}^{k_2+m-1}\psi^{(l+m-1)}_{k_1,k}(y)\underset{i-1+j}{\circ} \varphi^{(n)}_{t}(c)=\psi_{k_1,k_2+j+t-2}^{(l+m+n-2)}(y\underset{2}{\cdot} c), & i=k_2,\\
{\tiny
\begin{cases}
0, & i-1+j\neq k_1,k_2\\
\psi_{k_1+t-1,k_2+n-1}^{(l+m+n-2)}(y\underset{1}{\cdot} c), & i-1+j=k_1,(impossible)\\
\psi_{k_1,k_2+t-1}^{(l+m+n-2)}(y\underset{2}{\cdot} c), & i-1+j=k_2,(impossible).
\end{cases}
}& k_2<i\le l.
\end{cases}}
\end{aligned}
$$
and
$$\begin{aligned}
{\text{RHS of \eqref{E1.1.1}}}
&=\psi^{(l)}_{k_1k_2}(y)\underset{i}{\circ}(\1_m
\underset{j}{\circ} \varphi^{(n)}_{t}(c))\\
&\underset{\text{by (C32)}}{=}\psi^{(l)}_{k_1k_2}(y)\underset{i}{\circ}\varphi^{(m+n-1)}_{j+t-1}(c)\\
&\underset{\text{by (C38)}}{=}
{\tiny
\begin{cases}
0& i\neq k_1,k_2,\\
\psi_{k_1+j+t-2,k_2+m+n-2}^{(l+m+n-2)}(y\underset{1}{\cdot} c)
&i=k_1,\\
\psi_{k_1,k_2+j+t-2}^{(l+m+n-2)}(y\underset{2}{\cdot} c)
&i=k_2.
\end{cases}
}
\end{aligned}
$$
Hence \eqref{E1.1.1} holds.

{\sf Case 5: $\lambda \in \Psi$, $\mu\in \Phi$, and $\nu\in {\mathbb A}$. }
Write $\lambda =\psi^{(l)}_{k_1k_2}(y)$, $\mu=\varphi^{(m)}_{s}(b)$,
and $\nu=\1_n$. Then
$$\begin{aligned}
{\text{LHS of \eqref{E1.1.1}}}
&=(\psi^{(l)}_{k_1k_2}(y)\underset{i}{\circ} \varphi^{(m)}_{s}(b))
\underset{i-1+j}{\circ} \1_n\\
&\underset{\text{by (C38)}}{=}{\tiny
\begin{cases}
0 & i\neq k_1,k_2,\\
\psi_{k_1+s-1,k_2+m-1}^{(l+m-1)}(y\underset{1}{\cdot}b) & i=k_1,\\
\psi_{k_1,k_2+s-1}^{(l+m-1)}(y\underset{2}{\cdot}b) & i=k_2,
\end{cases}} \qquad \underset{i-1+j}{\circ} \1_n \\
&\underset{\text{by (C37)}}{=}{\tiny
\begin{cases}
0 & i\neq k_1, k_2,\\
{\tiny
\begin{cases}
\psi^{(l+m+n-2)}_{k_1+s-1+n-1, k_2+m-1+n-1}(y\underset{1}{\cdot}b), & 1\le i-1+j<k_1+s-1, \\
\sum\limits_{k=k_1+s-1}^{k_1+s-1+n-1} \psi^{(l+m+n-2)}_{k, k_2+m-1+n-1}(y\underset{1}{\cdot}b), & i-1+j=k_1+s-1, \\
\psi^{(l+m+n-2)}_{k_1+s-1, k_2+m-1+n-1}(y\underset{1}{\cdot}b), & k_1+s-1<i-1+j<k_2+m-1,\\
\sum\limits_{k=k_2+m-1}^{k_2+m-1+n-1} \psi^{(l+m+n-2)}_{k_1+s-1, k}(y\underset{1}{\cdot}b), & i-1+j=k_2+m-1,(impossible)\\
\psi^{(l+m+n-2)}_{k_1+s-1,k_2+m-1}(y\underset{1}{\cdot}b), & k_2+m-1<i-1+j\le l+m-1,(impossible).
\end{cases}
} & i=k_1,\\
{\tiny
\begin{cases}
\psi^{(l+m+n-2)}_{k_1+n-1, k_2+s-1+n-1}(y\underset{2}{\cdot}b), & 1\le i-1+j<k_1,(impossible) \\
\sum\limits_{k=k_1}^{k_1+n-1} \psi^{(l+m+n-2)}_{k, k_2+s-1+n-1}(y\underset{2}{\cdot}b), & i-1+j=k_1, (impossible)\\
\psi^{(l+m+n-2)}_{k_1, k_2+s-1+n-1}(y\underset{2}{\cdot}b), & k_1<i-1+j<k_2+s-1,\\
\sum\limits_{k=k_2+s-1}^{k_2+s-1+n-1} \psi^{(l+m+n-2)}_{k_1, k}(y\underset{2}{\cdot}b), & i-1+j=k_2+s-1,\\
\psi^{(l+m+n-2)}_{k_1k_2+s-1}(y\underset{2}{\cdot}b), & k_2+s-1<i-1+j\le l+m-1.
\end{cases}
} & i=k_2,\\
\end{cases}}
\end{aligned}
$$
and
$$\begin{aligned}
{\text{RHS of \eqref{E1.1.1}}}
&=\psi^{(l)}_{k_1k_2}(y)\underset{i}{\circ} (\varphi^{(m)}_{s}(b)
\underset{j}{\circ} \1_n)\\
&\underset{\text{by (C34)}}{=}
\psi^{(l)}_{k_1k_2}(y)\underset{i}{\circ}
{\tiny \begin{cases}
\varphi^{(m+n-1)}_s(b), & s<j,\\
\sum\limits_{k=s}^{s+n-1}\varphi^{(m+n-1)}_{k}(b)
+\sum\limits_{s\le p_1<p_2\le s+n-1}\psi^{(m+n-1)}_{p_1p_2}(f(b)), & s=j, \\
\varphi^{(m+n-1)}_{s+n-1}(b), & s>j.
\end{cases}}\\
&\underset{\text{by (C38)}}{=}{\tiny
\begin{cases}
{\tiny
\begin{cases}
0 & i\neq k_1,k_2,\\
\psi_{k_1+s-1,k_2+m+n-2}^{(l+m+n-2)}(y\underset{1}{\cdot}b) & i=k_1,\\
\psi_{k_1,k_2+s-1}^{(l+m+n-2)}(y\underset{2}{\cdot}b) & i=k_2,
\end{cases}}& s<j\\
{\tiny
\begin{cases}
0 & i\neq k_1,k_2,\\
\sum\limits_{k=s}^{s+n-1}\psi_{k_1+k-1,k_2+m+n-2}^{(l+m+n-2)}(y\underset{1}{\cdot}b) & i=k_1,\\
\sum\limits_{k=s}^{s+n-1}\psi_{k_1,k_2+k-1}^{(l+m+n-2)}(y\underset{2}{\cdot}b) & i=k_2,
\end{cases}}& s=j\\
{\tiny
\begin{cases}
0 & i\neq k_1,k_2,\\
\psi_{k_1+s+n-2,k_2+m+n-2}^{(l+m+n-2)}(y\underset{1}{\cdot}b) & i=k_1,\\
\psi_{k_1,k_2+s+n-2}^{(l+m+n-2)}(y\underset{2}{\cdot}b) & i=k_2,
\end{cases}}& s>j
\end{cases}}
\end{aligned}
$$
which implies that \eqref{E1.1.1} holds.

{\sf Case 6: $\mu \in \Psi$, $\lambda\in {\mathbb A}$, and $\nu\in \Phi$. }
Write $\lambda =\1_l$, $\mu=\psi^{(m)}_{k_1k_2}(y)$,
and $\nu=\varphi^{(n)}_{t}(c)$. Then
$$\begin{aligned}
{\text{LHS of \eqref{E1.1.1}}}
&=(\1_l\underset{i}{\circ}\psi_{k_1k_2}^{(m)}(y))
\underset{i-1+j}{\circ} \varphi^{(n)}_{t}(c)\\
&\underset{\text{by (C33)}}{=}
\psi_{i+k_1-1,i+k_2-1}^{(l+m-1)}(y) \underset{i-1+j}{\circ} \varphi^{(n)}_{t}(c)\\
&\underset{\text{by (C38)}}{=}
{\tiny
\begin{cases}
0 & j\neq k_1,k_2,\\
\psi_{i+k_1+t-2,i+k_2+n-2}^{(l+m+n-2)}(y\underset{1}{\cdot}c) & j=k_1,\\
\psi_{i+k_1-1,i+k_2+t-2}^{(l+m+n-2)}(y\underset{2}{\cdot}c) & j=k_2,
\end{cases}}
\end{aligned}
$$
and
$$\begin{aligned}
{\text{RHS of \eqref{E1.1.1}}}
&=\1_l\underset{i}{\circ} (\psi_{k_1k_2}^{(m)}(y)
\underset{j}{\circ} \varphi^{(n)}_{t}(c))\\
&\underset{\text{by (C38)}}{=}
\1_l\underset{i}{\circ}
{\tiny
\begin{cases}
0& j\neq k_1,k_2,\\
\psi_{k_1+t-1,k_2+n-1}^{(m+n-1)}(y\underset{1}{\cdot} c) & j=k_1,\\
\psi_{k_1,k_2+t-1}^{(m+n-1)}(y\underset{2}{\cdot} c) & j=k_2,
\end{cases}
}\\
&\underset{\text{by (C33)}}{=}
{\tiny
\begin{cases}
0& j\neq k_1,k_2,\\
\psi_{i+k_1+t-2,i+k_2+n-2}^{(l+m+n-2)}(y\underset{1}{\cdot} c) & j=k_1,\\
\psi_{i+k_1-1,i+k_2+t-2}^{(l+m+n-2)}(y\underset{2}{\cdot} c) & j=k_2,
\end{cases}
}
\end{aligned}
$$
Hence \eqref{E1.1.1} holds.

{\sf Case 7: $\mu \in \Psi$, $\lambda\in \Phi$, and $\nu\in {\mathbb A}$. }
Write $\lambda =\varphi^{(l)}_{r}(a)$, $\mu=\psi^{(m)}_{k_1k_2}(y)$,
and $\nu=\1_n$. Then
$$\begin{aligned}
{\text{LHS of \eqref{E1.1.1}}}
&=(\varphi^{(l)}_{r}(a)\underset{i}{\circ}\psi_{k_1k_2}^{(m)}(y))
\underset{i-1+j}{\circ} \1_n\\
&\underset{\text{by (C36)}}{=}
{\tiny
\begin{cases}
0& i\neq r,\\
\psi_{k_1+r-1,k_2+r-1}^{(l+m-1)}(ay)& i=r
\end{cases}
} \quad \underset{i-1+j}{\circ} \1_n\\
&\underset{\text{by (C37)}}{=}
{\tiny
\begin{cases}
0& i\neq r,\\
\psi^{(l+m+n-2)}_{k_1+r-1+n-1, k_2+r-1+n-1}(ay), & 1\le r-1+j<r+k_1-1, \\
\sum\limits_{k=k_1+r-1}^{k_1+r-1+n-1} \psi^{(l+m+n-2)}_{k, k_2+r-1+n-1}(ay), & r-1+j=r+k_1-1, \\
\psi^{(l+m+n-2)}_{k_1+r-1, k_2+r-1+n-1}(ay), & r+k_1-1<r-1+j<r+k_2-1,\\
\sum\limits_{k=k_2+r-1}^{k_2+r-1+n-1} \psi^{(l+m+n-2)}_{k_1+r-1, k}(ay), & r-1+j=r+k_2-1,\\
\psi^{(l+m+n-2)}_{k_1+r-1,k_2+r-1}(ay), & r+k_2-1<r-1+j\le l+m-1.
\end{cases}
}
\end{aligned}
$$
and
$$\begin{aligned}
{\text{RHS of \eqref{E1.1.1}}}
&=\varphi^{(l)}_{r}(a)\underset{i}{\circ} (\psi_{k_1k_2}^{(m)}(y)
\underset{j}{\circ} \1_n)\\
&\underset{\text{by (C37)}}{=}
\varphi^{(l)}_{r}(a)\underset{i}{\circ}
{\tiny
\begin{cases}
\psi^{(m+n-1)}_{k_1+n-1, k_2+n-1}(y), & 1\le j<k_1, \\
\sum\limits_{k=k_1}^{k_1+n-1} \psi^{(m+n-1)}_{k, k_2+n-1}(y), & j=k_1, \\
\psi^{(m+n-1)}_{k_1, k_2+n-1}(y), & k_1<j<k_2,\\
\sum\limits_{k=k_2}^{k_2+n-1} \psi^{(m+n-1)}_{k_1, k}(y), & j=k_2,\\
\psi^{(m+n-1)}_{k_1k_2}(y), & k_2<j\le m.
\end{cases}
}\\
&\underset{\text{by (C36)}}{=}
{\tiny
\begin{cases}
{\tiny
\begin{cases}
0& i\neq r,\\
\psi_{r+k_1+n-2, r+k_2+n-2}^{(l+m+n-2)}(ay)& i=r
\end{cases}
}& 1\le j<k_1,\\
{\tiny
\begin{cases}
0& i\neq r,\\
\sum\limits_{k=k_1}^{k_1+n-1} \psi_{r+k-1, r+k_2+n-2}^{(l+m+n-2)}(ay)& i=r
\end{cases}
}& j=k_1,\\
{\tiny
\begin{cases}
0& i\neq r,\\
\psi_{r+k_1-1, r+k_2+n-2}^{(l+m+n-2)}(ay)& i=r
\end{cases}
}& k_1<j<k_2,\\
{\tiny
\begin{cases}
0& i\neq r,\\
\sum\limits_{k=k_2}^{k_2+n-1} \psi_{r+k_1-1, r+k-1}^{(l+m+n-2)}(ay)& i=r
\end{cases}
}& j=k_2,\\
{\tiny
\begin{cases}
0& i\neq r,\\
\psi_{r+k_1-1,r+k_2-1}^{(l+m+n-2)}(ay)& i=r
\end{cases}
}& k_2<j\le m.
\end{cases}
}
\end{aligned}
$$
Hence \eqref{E1.1.1} holds.

{\sf Case 8: $\nu \in \Psi$, $\lambda\in {\mathbb A}$, and $\mu\in \Phi$.}
Write $\lambda=\1_l$, $\mu=\varphi^{(m)}_k(a)$, and
$\nu=\psi^{(n)}_{k_1k_2}(y)$. Then
$$\begin{aligned}
{\text{LHS of \eqref{E1.1.1}}}
&=(\1_l\underset{i}{\circ} \varphi^{(m)}_k(a))
  \underset{i-1+j}{\circ} \psi^{(n)}_{k_1k_2}(y)\\
&\underset{\text{by (C32)}}{=}
\varphi^{(l+m-1)}_{i+k-1}(a)\underset{i-1+j}{\circ} \psi^{(n)}_{k_1k_2}(y)\\
&\underset{\text{by (C36)}}{=}
{\tiny \begin{cases}
0 & i+k-1\neq i-1+j,\\
\psi^{(l+m+n-2)}_{k_1+i+j-2,k_2+i+j-2}(ay)&i+k-1=i-1+j,
\end{cases}
}
\end{aligned}
$$
and
$$\begin{aligned}
{\text{RHS of \eqref{E1.1.1}}}
&=\1_l\underset{i}{\circ} (\varphi^{(m)}_k(a)
  \underset{j}{\circ} \psi^{(n)}_{k_1k_2}(y))\\
&\underset{\text{by (C36)}}{=}
\1_l\underset{i}{\circ} \quad
{\tiny \begin{cases}
0 & k\neq j,\\
\psi^{(m+n-1)}_{k_1+j-1,k_2+j-1}(ay)& k= j
\end{cases}}\\
&\underset{\text{by (C32)}}{=}
{\tiny \begin{cases}
0 & k\neq j,\\
\psi^{(l+m+n-2)}_{k_1+j+i-2,k_2+j+i-2}(ay)& k=j.
\end{cases}}
\end{aligned}
$$
Hence \eqref{E1.1.1} holds.

{\sf Case 9: $\nu \in \Psi$, $\lambda\in \Phi$, and $\mu\in {\mathbb A}$.}
Write $\lambda=\varphi^{(l)}_{r}(a)$, $\1_m$,
and $\nu=\psi_{k_1k_2}^{(n)}(y)$. Then
$$\begin{aligned}
{\text{LHS of \eqref{E1.1.1}}}
&=(\varphi^{(l)}_{r}(a)\underset{i}{\circ}\1_m)
\underset{i-1+j}{\circ} \psi_{k_1k_2}^{(n)}(y)\\
&\underset{\text{by (C34)}}{=}
{\tiny \begin{cases}
\varphi^{(l+m-1)}_r(a), & r<i,\\
\sum\limits_{k=r}^{r+m-1}\varphi^{(l+m-1)}_{k}(a)
+\sum\limits_{r\le p_1<p_2\le r+m-1}\psi^{(l+m-1)}_{p_1p_2}(f(a)), & r=i, \\
\varphi^{(l+m-1)}_{r+m-1}(a), & r>i.
\end{cases}}
\quad \underset{i-1+j}{\circ} \psi_{k_1k_2}^{(n)}(y)\\
&\underset{\text{by (C36)}}{=}
{\tiny
\begin{cases}
{\tiny \begin{cases}
0 & r\neq i-1+j,\\
\psi^{(l+m+n-2)}_{r+k_1-1,r+k_2-1}(ay)& r=i-1+j,(impossible).
\end{cases}}& r<i\\
\psi^{(l+m+n-2)}_{i-1+j+k_1-1,i-1+j+k_2-1}(ay)& r=i\\
{\tiny \begin{cases}
0 & r+m-1\neq i-1+j,\\
\psi^{(l+m+n-2)}_{r+m-1+k_1-1,r+m-1+k_1-1}(ay)& r+m-1=i-1+j,(impossible).
\end{cases}}& r>i
\end{cases}
}\\
\end{aligned}
$$
And
$$\begin{aligned}
{\text{RHS of \eqref{E1.1.1}}}
&=\varphi^{(l)}_{r}(a)\underset{i}{\circ}(\1_m
\underset{j}{\circ} \psi_{k_1k_2}^{(n)}(y))\\
&\underset{\text{by (C33)}}{=}
\varphi^{(l)}_{r}(a)\underset{i}{\circ}
\psi_{j+k_1-1,j+k_2-1}^{(m+n-1)}(y)\\
&\underset{\text{by (C36)}}{=}
{\tiny
\begin{cases}
0& r\neq i,\\
\psi_{r+j+k_1-2,r+j+k_2-2}^{(l+m+n-2)}(ay))& r=i.
\end{cases}
}
\end{aligned}
$$
Hence \eqref{E1.1.1} holds.

{\sf Case 10: $\lambda,\mu,\nu\in \Phi$.}
Write $\lambda=\varphi_{r}^{(l)}(a)$ for $a\in {\bar A}$, $\mu=\varphi_{s}^{(m)}(b)$
and $\nu=\varphi^{(n)}_t(c)$. Then
$$\begin{aligned}
{\text{LHS of \eqref{E1.1.1}}}
&=(\varphi_{r}^{(l)}(a)\underset{i}{\circ} \varphi_{s}^{(m)}(b)) \underset{i-1+j}{\circ} \varphi^{(n)}_t(c)\\
&\underset{\text{by (C35)}}{=}
{\tiny \begin{cases}
\psi^{(l+m-1)}_{r, i+s-1}(g(a, b)), & r<i, \\
\begin{array}{ll}
\sum\limits_{k=r}^{r+s-2} \psi^{(l+m-1)}_{k, r+s-1}(g(a, b)
+f(a)\underset{2}\cdot b)+\varphi^{(l+m-1)}_{r+s-1}(ab)\\
+\sum\limits_{k=r+s}^{r+m-1} \psi^{(l+m-1)}_{r+s-1, k}(g(b, a)
+f(a)\underset{1}\cdot b),
\end{array}& r=i, \\
\psi^{(l+m-1)}_{i+s-1, r+m-1}(g(b, a)), & r>i.
\end{cases}} \quad \underset{i-1+j}{\circ} \varphi^{(n)}_t(c)\\
&\underset{\text{by (C38)}}{=}
{\tiny
\begin{cases}
{\tiny
\begin{cases}
0, & j\neq s,\\
\psi^{(l+m+n-2)}_{r,i+s-1+t-1}(g(a, b)\underset{2}\cdot c), & j= s.
\end{cases}
} & r<i,\\
Xterm,& r=i,\\
{\tiny
\begin{cases}
0, & j\neq s,\\
\psi^{(l+m+n-2)}_{i+s-1+t-1,r+m-1+n-1}(g(b, a)\underset{1}\cdot c), & j= s.
\end{cases}
} & r>i.
\end{cases}}
\end{aligned}
$$
where
$$\begin{aligned}
Xterm&=
(\sum\limits_{k=i}^{i+s-2} \psi^{(l+m-1)}_{k, i+s-1}(g(a, b)
+f(a)\underset{2}\cdot b)+\varphi^{(l+m-1)}_{i+s-1}(ab)
+\sum\limits_{k=i+s}^{i+m-1} \psi^{(l+m-1)}_{i+s-1, k}(g(b, a)
+f(a)\underset{1}\cdot b))
\underset{i-1+j}{\circ} \varphi^{(n)}_t(c)\\
&\underset{\text{by (C35),(C38)}}{=}{\tiny
\begin{cases}
\psi^{(l+m-1+n-1)}_{i+s-1, i+j-1+t-1}(g(ab, c))+\psi^{(l+m-1+n-1)}_{i+s-1, i+j-1+t-1}(g(b, a)
+f(a)\underset{1}\cdot b)\underset{2}\cdot c,& s<j,\\
{\tiny
\begin{cases}
\sum\limits_{k=i}^{i+s-2} \psi^{(l+m-1+n-1)}_{k, i+s-1+t-1}(g(a, b)
+f(a)\underset{2}\cdot b)\underset{2}\cdot c\\
+\sum\limits_{k=i+s-1}^{i+s-1+t-2} \psi^{(l+m-1+n-1)}_{k, i+s-1+t-1}(g(ab, c)+f(ab)\underset{2}\cdot c)
+\varphi^{(l+m-1+n-1)}_{i+s-1+t-1}((ab)c)\\
+\sum\limits_{k=i+s-1+t}^{i+s-1+n-1} \psi^{(l+m-1+n-1)}_{i+s-1+t-1, k}(g(c, ab)+f(ab)\underset{1}\cdot c)\\
+\sum\limits_{k=i+s}^{i+m-1} \psi^{(l+m-1+n-1)}_{i+s-1+t-1, k+n-1}(g(b, a)
+f(a)\underset{1}\cdot b)\underset{1}\cdot c
\end{cases}
} & s=j,\\
\psi^{(l+m-1+n-1)}_{i+j-1+t-1, i+s-1+n-1}(g(a, b)+f(a)\underset{2}\cdot b)\underset{1}\cdot c+\psi^{(l+m-1+n-1)}_{i+j-1+t-1, i+s-1+n-1}(g(c, ab)), & s> j.
\end{cases}
}
\end{aligned}
$$
and
$$\begin{aligned}
{\text{RHS of \eqref{E1.1.1}}}
&=\varphi_{r}^{(l)}(a)\underset{i}{\circ} (\varphi_{s}^{(m)}(b) \underset{j}{\circ} \varphi^{(n)}_t(c))\\
&\underset{\text{by (C35)}}{=}\varphi_{r}^{(l)}(a)\underset{i}{\circ} {\tiny \begin{cases}
\psi^{(m+n-1)}_{s, j+t-1}(g(b, c)), & s<j, \\
\begin{array}{ll}
\sum\limits_{k=s}^{s+t-2} \psi^{(m+n-1)}_{k, s+t-1}(g(b, c)
+f(b)\underset{2}\cdot c)+\varphi^{(m+n-1)}_{s+t-1}(bc)\\
+\sum\limits_{k=s+t}^{s+n-1} \psi^{(m+n-1)}_{s+t-1, k}(g(c, b)
+f(b)\underset{1}\cdot c),
\end{array}& s=j, \\
\psi^{(m+n-1)}_{j+t-1, s+n-1}(g(c, b)), & s>j.
\end{cases}}\\
&\underset{\text{by (C36)}}{=}
{\tiny \begin{cases}
{\tiny
\begin{cases}
0, & r\neq i,\\
\psi^{(l+m+n-2)}_{r+s-1,r+j+t-1-1}(ag(b,c)), & r= i.
\end{cases}
} & s<j,\\
Yterm& s=j,\\
{\tiny
\begin{cases}
0, & r\neq i,\\
\psi^{(l+m+n-2)}_{r+j+t-1-1,r+s+n-1-1}(ag(c, b)), & r= i.
\end{cases}
} & s>j.
\end{cases}
}
\end{aligned}
$$
where
$$\begin{aligned}
Yterm&\underset{s=j}{=}
\varphi_{r}^{(l)}(a)\underset{i}{\circ}(\sum\limits_{k=s}^{s+t-2} \psi^{(m+n-1)}_{k, s+t-1}(g(b, c)
+f(b)\underset{2}\cdot c)+\varphi^{(m+n-1)}_{s+t-1}(bc)
+\sum\limits_{k=s+t}^{s+n-1} \psi^{(m+n-1)}_{s+t-1, k}(g(c, b)
+f(b)\underset{1}\cdot c))\\
&\underset{\text{by (C35),(C36)}}{=}
{\tiny
\begin{cases}
\psi^{(l+m-1+n-1)}_{r, i+s+t-1-1}g(a, bc) & r<i,\\
{\tiny
\begin{cases}
\sum\limits_{k=s}^{s+t-2} \psi^{(l+m-1+n-1)}_{i+k-1, i+s+t-1-1}(a(g(b, c)
+f(b)\underset{2}\cdot c))\\
+\sum\limits_{k=i}^{i+s+t-1-2} \psi^{(l+m-1+n-1)}_{k, i+s+t-1-1}(g(a, bc)
+f(a)\underset{2}\cdot bc)+\varphi^{(l+m-1+n-1)}_{i+s+t-1-1}(a(bc))\\
+\sum\limits_{k=i+s+t-1}^{i+m+n-1-1} \psi^{(l+m-1+n-1)}_{i+s+t-1-1, k}(g(bc, a)
+f(a)\underset{1}\cdot bc)\\
+\sum\limits_{k=s+t}^{s+n-1} \psi^{(l+m-1+n-1)}_{i+s+t-1-1, i+k-1}(a(g(c, b)
+f(b)\underset{1}\cdot c))
\end{cases}
}& r=i,\\
 \psi^{(l+m-1+n-1)}_{i+s+t-1-1, r+m-1+n-1}g(bc, a), & r> i.
\end{cases}
}
\end{aligned}
$$
Hence \eqref{E1.1.1} holds.

{\sf Case 11: $\lambda \in {\mathbb A}$ and $\mu,\nu\in \Phi$. }
Write $\lambda=\1_l$, $\mu=\varphi^{(m)}_s(b)$, and
$\nu=\varphi^{(n)}_t(c)$. Then
$$\begin{aligned}
{\text{LHS of \eqref{E1.1.1}}}
&=(\1_l\underset{i}{\circ} \varphi^{(m)}_s(b))
  \underset{i-1+j}{\circ} \varphi^{(n)}_t(c)\\
&\underset{\text{by (C32)}}{=}
\varphi^{(l+m-1)}_{i+s-1}(b)\underset{i-1+j}{\circ} \varphi^{(n)}_t(c)\\
&\underset{\text{by (C35)}}{=}
{\tiny \begin{cases}
\psi^{(l+m+n-2)}_{i+s-1, i-1+j+t-1}(g(b, c)), & i+s-1<i-1+j, \\
\begin{array}{ll}
\sum\limits_{k=i+s-1}^{i+s-1+t-2} \psi^{(l+m+n-2)}_{k, i+s-1+t-1}(g(b,c)
+f(b)\underset{2}\cdot c)+\varphi^{(l+m+n-2)}_{i+s-1+t-1}(bc)\\
+\sum\limits_{k=i+s-1+t}^{i+s-1+n-1} \psi^{(l+m+n-2)}_{i+s-1+t-1, k}(g(c, b)
+f(b)\underset{1}\cdot c),
\end{array}&  i+s-1=i-1+j, \\
\psi^{(l+m+n-2)}_{i-1+j+t-1, i+s-1+n-1}(g(c, b)), &  i+s-1>i-1+j.
\end{cases}
}
\end{aligned}
$$
and
$$\begin{aligned}
{\text{RHS of \eqref{E1.1.1}}}
&=\1_l\underset{i}{\circ} (\varphi^{(m)}_s(b)
  \underset{j}{\circ} \varphi^{(n)}_t(c))\\
&\underset{\text{by (C35)}}{=}
\1_l\underset{i}{\circ} \quad
{\tiny \begin{cases}
\psi^{(m+n-1)}_{s, j+t-1}(g(b, c)), & s<j, \\
\begin{array}{ll}
\sum\limits_{k=s}^{s+t-2} \psi^{(m+n-1)}_{k, s+t-1}(g(b, c)
+f(b)\underset{2}\cdot c)+\varphi^{(m+n-1)}_{s+t-1}(bc)\\
+\sum\limits_{k=s+t}^{s+n-1} \psi^{(m+n-1)}_{s+t-1, k}(g(c, b)
+f(b)\underset{1}\cdot c),
\end{array}& s=j, \\
\psi^{(m+n-1)}_{j+t-1, s+n-1}(g(c, b)), & s>j.
\end{cases}}\\
&\underset{\text{by (C32),(C33)}}{=}
{\tiny \begin{cases}
\psi^{(l+m+n-2)}_{i+s-1, i+j+t-2}(g(b, c)), & s<j, \\
\begin{array}{ll}
\sum\limits_{k=s}^{s+t-2} \psi^{(l+m+n-2)}_{i+k-1, i+s+t-2}(g(b, c)
+f(b)\underset{2}\cdot c)+\varphi^{(l+m+n-2)}_{i+s+t-2}(bc)\\
+\sum\limits_{k=s+t}^{s+n-1} \psi^{(l+m+n-2)}_{i+s+t-2, i+k-1}(g(c, b)
+f(b)\underset{1}\cdot c),
\end{array}& s=j, \\
\psi^{(l+m+n-2)}_{i+j+t-2, i+s+n-2}(g(c, b)), & s>j.
\end{cases}}
\end{aligned}
$$
Hence \eqref{E1.1.1} holds.

{\sf Case 12: $\mu \in {\mathbb A}$ and $\lambda,\nu\in \Phi$. }
Write $\lambda=\varphi_{r}^{(l)}(a)$, $\mu=\1_m$
and $\nu=\varphi_{t}^{(n)}(c)$. Then
$$\begin{aligned}
{\text{LHS of \eqref{E1.1.1}}}
&=(\varphi_{r}^{(l)}(a)\underset{i}{\circ} \1_m) \underset{i-1+j}{\circ} \varphi_{t}^{(n)}(c)\\
&\underset{\text{by (C34)}}{=}
{\tiny \begin{cases}
\varphi^{(l+m-1)}_r(a), & r<i,\\
\sum\limits_{w=r}^{r+m-1}\varphi^{(l+m-1)}_{w}(a)
+\sum\limits_{r\le k_1<k_2\le r+m-1}\psi^{(l+m-1)}_{k_1k_2}(f(a)), & r=i, \\
\varphi^{(l+m-1)}_{r+m-1}(a), & r>i.
\end{cases}} \quad \underset{i-1+j}{\circ} \varphi_{t}^{(n)}(c)\\
&\underset{\text{by (C35)}}{=}
{\tiny
\begin{cases}
{\tiny
\begin{cases}
\psi^{(l+m+n-2)}_{r, i-1+j+t-1}(g(a,c)), & r<i-1+j, \\
(impossible), & others.
\end{cases}
} & r<i.\\
Xterm& r=i,\\
{\tiny
\begin{cases}
\psi^{(l+m+n-2)}_{i-1+j+t-1, r+m-1+n-1}(g(c,a)), & r+m-1>i-1+j, \\
(impossible), & others.
\end{cases}
} & r>i.\\
\end{cases}}
\end{aligned}
$$
where
$$\begin{aligned}
Xterm&=
\sum\limits_{w=r}^{r+m-1}\varphi^{(l+m-1)}_{w}(a)\underset{i-1+j}{\circ} \varphi_{t}^{(n)}(c)
+\sum\limits_{r\le k_1<k_2\le r+m-1}\psi^{(l+m-1)}_{k_1k_2}(f(a))\underset{i-1+j}{\circ} \varphi_{t}^{(n)}(c)\\
&\underset{\text{by (r=i)}}{=}{\tiny
\begin{cases}
\sum\limits_{w=i}^{i-1+j-1}\varphi^{(l+m-1)}_{w}(a)\underset{i-1+j}{\circ} \varphi_{t}^{(n)}(c)
+\varphi^{(l+m-1)}_{i-1+j}(a)\underset{i-1+j}{\circ} \varphi_{t}^{(n)}(c)
+\sum\limits_{w=i+j}^{i+m-1}\varphi^{(l+m-1)}_{w}(a)\underset{i-1+j}{\circ}\varphi_{t}^{(n)}(c) \\
+\sum\limits_{k=i+j}^{i+m-1}\psi^{(l+m-1)}_{i-1+j,k}(f(a))\underset{i-1+j}{\circ} \varphi_{t}^{(n)}(c)+
\sum\limits_{k=i}^{i-1+j-1}\psi^{(l+m-1)}_{k,i-1+j}(f(a))\underset{i-1+j}{\circ} \varphi_{t}^{(n)}(c)
\end{cases}
}\\
&\underset{\text{by (C35),(C38)}}{=}
{\tiny \begin{cases}
\sum\limits_{w=i}^{i-1+j-1}\psi^{(l+m+n-2)}_{w, i-1+j+t-1}(g(a,c))
+\sum\limits_{k=i-1+j}^{i-1+j+t-2} \psi^{(l+m+n-2)}_{k, i-1+j+t-1}(g(a, c)+f(a)\underset{2}\cdot c)\\
+\varphi^{(l+m+n-2)}_{i-1+j+t-1}(ac)
+\sum\limits_{k=i-1+j+t}^{i-1+j+n-1} \psi^{(l+m+n-2)}_{i-1+j+t-1, k}(g(c, a)
+f(a)\underset{1}\cdot c)\\
+\sum\limits_{w=i+j}^{i+m-1}\psi^{(l+m+n-2)}_{i-1+j+t-1, w+n-1}(g(c,a)) \\
+\sum\limits_{k=i+j}^{i+m-1}\psi^{(l+m-1)}_{i-1+j+t-1,k+n-1}(f(a)\underset{1}\cdot c)+
\sum\limits_{k=i}^{i-1+j-1}\psi^{(l+m-1)}_{k,i-1+j+t-1}(f(a)\underset{2}\cdot c)
\end{cases}
}
\end{aligned}
$$
and
$$\begin{aligned}
{\text{RHS of \eqref{E1.1.1}}}
&=\varphi_{r}^{(l)}(a)\underset{i}{\circ} (\1_m \underset{j}{\circ} \varphi_{t}^{(n)}(c))\\
&\underset{\text{by (C32)}}{=}\varphi_{r}^{(l)}(a)\underset{i}{\circ} \varphi_{j+t-1}^{(m+n-1)}(c) \\
&\underset{\text{by (C35)}}{=}
{\tiny \begin{cases}
\psi^{(l+m+n-2)}_{r, i+j+t-1-1}(g(a,c)), & r<i, \\
\begin{array}{ll}
\sum\limits_{k=r}^{r+j+t-1-2} \psi^{(l+m+n-2)}_{k, r+j+t-1-1}(g(a, c)
+f(a)\underset{2}\cdot c)+\varphi^{(l+m+n-2)}_{r+j+t-1-1}(ac)\\
+\sum\limits_{k=r+j+t-1}^{r+m+n-1-1} \psi^{(l+m+n-2)}_{r+j+t-1-1, k}(g(c, a)
+f(a)\underset{1}\cdot c),
\end{array}& r=i, \\
\psi^{(l+m+n-2)}_{i+j+t-1-1, r+m+n-1-1}(g(c, a)), & r>i.
\end{cases}
}
\end{aligned}
$$
Hence \eqref{E1.1.1} holds.

{\sf Case 13: $\nu \in {\mathbb A}$ and $\lambda,\mu\in \Phi$. }
Write $\lambda=\varphi_{r}^{(l)}(a)$ for $a\in {\bar A}$, $\mu=\varphi_{s}^{(m)}(b)$
and $\nu=\1_n$. Then
$$\begin{aligned}
{\text{LHS of \eqref{E1.1.1}}}
&=(\varphi_{r}^{(l)}(a)\underset{i}{\circ} \varphi_{s}^{(m)}(b)) \underset{i-1+j}{\circ} \1_n\\
&\underset{\text{by (C35)}}{=}
{\tiny \begin{cases}
\psi^{(l+m-1)}_{r, i+s-1}(g(a, b)), & r<i, \\
\begin{array}{ll}
\sum\limits_{k=r}^{r+s-2} \psi^{(l+m-1)}_{k, r+s-1}(g(a, b)
+f(a)\underset{2}\cdot b)+\varphi^{(l+m-1)}_{r+s-1}(ab)\\
+\sum\limits_{k=r+s}^{r+m-1} \psi^{(l+m-1)}_{r+s-1, k}(g(b, a)
+f(a)\underset{1}\cdot b),
\end{array}& r=i, \\
\psi^{(l+m-1)}_{i+s-1, r+m-1}(g(b, a)), & r>i.
\end{cases}} \quad \underset{i-1+j}{\circ} \1_n\\
&\underset{\text{by (C37)}}{=}
{\tiny
\begin{cases}
{\tiny
\begin{cases}
\psi^{(l+m+n-2)}_{r, i+s-1+n-1}(g(a, b), & j<s,\\
\sum\limits_{k=i+s-1}^{i+s-1+n-1} \psi^{(l+m+n-2)}_{r, k}(g(a, b), & j=s,\\
\psi^{(l+m+n-2)}_{r,i+s-1}(g(a, b), & j> s.
\end{cases}
}, & r<i,\\
Xterm,& r=i,\\
{\tiny
\begin{cases}
\psi^{(l+m+n-2)}_{i+s-1+n-1, r+m-1+n-1}(g(b, a)), & j<s, \\
\sum\limits_{k=i+s-1}^{i+s-1+n-1} \psi^{(l+m+n-2)}_{k, r+m-1+n-1}(g(b, a)), & j=s, \\
\psi^{(l+m+n-2)}_{i+s-1, r+m-1+n-1}(g(b, a)), & j>s
\end{cases}
}, & r>i.
\end{cases}}
\end{aligned}
$$
where
$$\begin{aligned}
Xterm&=
(\sum\limits_{k=i}^{i+s-2} \psi^{(l+m-1)}_{k, i+s-1}(g(a, b)
+f(a)\underset{2}\cdot b)+\varphi^{(l+m-1)}_{i+s-1}(ab)
+\sum\limits_{k=i+s}^{i+m-1} \psi^{(l+m-1)}_{i+s-1, k}(g(b, a)
+f(a)\underset{1}\cdot b),)
\underset{k-1+j}{\circ} \1_n\\
&\underset{\text{by (C34),(C37)}}{=}{\tiny
\begin{cases}
{\tiny
\begin{cases}
\sum\limits_{k=i}^{i+s-2+n-1} \psi^{(l+m-1+n-1)}_{k, i+s-1+n-1}(g(a, b)+f(a)\underset{2}\cdot b),\\
+\varphi^{(l+m-1+n-1)}_{i+s-1+n-1}(ab),\\
+\sum\limits_{k=i+s+n-1}^{i+m-1+n-1} \psi^{(l+m-1+n-1)}_{i+s-1+n-1, k}(g(b, a)+f(a)\underset{1}\cdot b)),
\end{cases}
} & j<s,\\
{\tiny
\begin{cases}
\sum\limits_{i\leq k_1<i+s-2,i+s-1\leq k_2<i+s-1+n-1}\psi^{(l+m-1+n-1)}_{k_1, k_2}(g(a, b)+f(a)\underset{2}\cdot b)\\
+\sum\limits_{k=i+s-1}^{i+s-1+n-1}\varphi^{(l+m-1+n-1)}_{k}(ab)+\sum\limits_{i+s-1\leq k_1<k_2<i+s-1+n-1} \psi^{(l+m-1+n-1)}_{k_1, k_2}f(ab)\\
+\sum\limits_{i+s-1\leq k_1<i+s-1+n-1,i+s+n-1\leq k_2<i+m-1+n-1} \psi^{(l+m-1+n-1)}_{k_1, k_2}(g(b, a)+f(a)\underset{1}\cdot b))
\end{cases}
} & j=s,\\
{\tiny
\begin{cases}
\sum\limits_{k=i}^{i+s-2} \psi^{(l+m-1+n-1)}_{k, i+s-1}(g(a, b)+f(a)\underset{2}\cdot b),\\
+\varphi^{(l+m-1+n-1)}_{i+s-1}(ab),\\
+\sum\limits_{k=i+s}^{i+m-1+n-1} \psi^{(l+m-1+n-1)}_{i+s-1, k}(g(b, a)+f(a)\underset{1}\cdot b)),
\end{cases}
}, & j> s.
\end{cases}
}
\end{aligned}
$$
and
$$\begin{aligned}
{\text{RHS of \eqref{E1.1.1}}}
&=\varphi_{r}^{(l)}(a)\underset{i}{\circ} (\varphi_{s}^{(m)}(b) \underset{j}{\circ} \1_n)\\
&\underset{\text{by (C34)}}{=}\varphi_{r}^{(l)}(a)\underset{i}{\circ} {\tiny \begin{cases}
\varphi^{(m+n-1)}_s(b), & s<j,\\
\sum\limits_{k=s}^{s+n-1}\varphi^{(m+n-1)}_{k}(b)
+\sum\limits_{s\le p_1<p_2\le s+n-1}\psi^{(m+n-1)}_{p_1p_2}(f(b)), & s=j, \\
\varphi^{(m+n-1)}_{s+n-1}(b), & s>j.
\end{cases}} \\
&\underset{\text{by (C35)}}{=}
{\tiny \begin{cases}
{\tiny \begin{cases}
\psi^{(l+m+n-2)}_{r, i+s-1}(g(a, b)), & r<i, \\
\begin{array}{ll}
\sum\limits_{k=r}^{r+s-2} \psi^{(l+m+n-2)}_{k, r+s-1}(g(a, b)
+f(a)\underset{2}\cdot b)+\varphi^{(l+m+n-2)}_{r+s-1}(ab)\\
+\sum\limits_{k=r+s}^{r+m+n-1-1} \psi^{(l+m+n-2)}_{r+s-1, k}(g(b, a)
+f(a)\underset{1}\cdot b),
\end{array}& r=i, \\
\psi^{(l+m+n-2)}_{i+s-1, r+m+n-1-1}(g(b, a)), & r>i.
\end{cases}} & s<j,\\
Yterm& s=j,\\
{\tiny \begin{cases}
\psi^{(l+m+n-2)}_{r, i+s+n-1-1}(g(a, b)), & r<i, \\
\begin{array}{ll}
\sum\limits_{k=r}^{r+s+n-1-2} \psi^{(l+m+n-2)}_{k, r+s+n-1-1}(g(a, b)
+f(a)\underset{2}\cdot b)+\varphi^{(l+m+n-2)}_{r+s+n-1-1}(ab)\\
+\sum\limits_{k=r+s+n-1}^{r+m+n-1-1} \psi^{(l+m+n-2)}_{r+s+n-1-1, k}(g(b, a)
+f(a)\underset{1}\cdot b),
\end{array}& r=i, \\
\psi^{(l+m+n-2)}_{i+s+n-1-1, r+m+n-1-1}(g(b, a)), & r>i.
\end{cases}} & s>j.
\end{cases}
}
\end{aligned}
$$
where
$$\begin{aligned}
Yterm&\underset{s=j}{=}
\varphi_{r}^{(l)}(a)\underset{i}{\circ}(\sum\limits_{k=s}^{s+n-1}\varphi^{(m+n-1)}_{k}(b)
+\sum\limits_{s\le p_1<p_2\le s+n-1}\psi^{(m+n-1)}_{p_1p_2}(f(b)))\\
&\underset{\text{by (C34),(C37)}}{=}
{\tiny
\begin{cases}
\sum\limits_{k=i+s-1}^{i+s-1+n-1} \psi^{(l+m-1+n-1)}_{r, k}g(a, b) & r<i,\\
{\tiny
\begin{cases}\sum\limits_{k=s}^{s+n-1}\sum\limits_{w=r}^{r+k-2} \psi^{(l+m-1+n-1)}_{w, r+k-1}(g(a, b)
+f(a)\underset{2}\cdot b)+\sum\limits_{k=s}^{s+n-1}\varphi^{(l+m-1+n-1)}_{r+k-1}(ab)\\
+\sum\limits_{k=s}^{s+n-1}\sum\limits_{w=r+k}^{r+m+n-1-1} \psi^{(l+m-1+n-1)}_{r+k-1, w}(g(b, a)
+f(a)\underset{1}\cdot b)\end{cases}
} & r=i,\\
\sum\limits_{k=i+s-1}^{i+s-1+n-1} \psi^{(l+m-1+n-1)}_{k, r+m-1+n-1}g(b, a), & r> i.
\end{cases}
}
\end{aligned}
$$
Hence \eqref{E1.1.1} holds.

{\sf Case 14: $\lambda \in \Phi$ and $\mu,\nu\in {\mathbb A}$. }
Write $\lambda=\varphi_{k}^{(l)}(a)$, $\mu=\1_m$
and $\nu=\1_n$. Then
$$\begin{aligned}
{\text{LHS of \eqref{E1.1.1}}}
&=(\varphi_{k}^{(l)}(a)\underset{i}{\circ} \1_m) \underset{i-1+j}{\circ} \1_n\\
&\underset{\text{by (C34)}}{=}
{\tiny \begin{cases}
\varphi^{(l+m-1)}_k(a), & k<i,\\
\sum\limits_{w=k}^{k+m-1}\varphi^{(l+m-1)}_{w}(a)
+\sum\limits_{k\le k_1<k_2\le k+m-1}\psi^{(l+m-1)}_{k_1k_2}(f(a)), & k=i, \\
\varphi^{(l+m-1)}_{k+m-1}(a), & k>i.
\end{cases}} \quad \underset{i-1+j}{\circ} \1_n\\
&\underset{\text{by (C34)}}{=}
{\tiny
\begin{cases}
\varphi^{(l+m+n-2)}_k(a), & k<i,\\
{\tiny
\begin{cases}
\sum\limits_{w=k}^{k+j-2} \varphi^{(l+m+n-2)}_{w}(a)\\
+\sum\limits_{v=k+j-1}^{k+j+n-2}\varphi^{(l+m+n-2)}_v(a)\\
+\sum\limits_{k+j-1\leq l_1<l_2\leq k+j+n-2}
\psi^{(l+m+n-2)}_{l_1l_2}(f(a))\\
+\sum\limits_{w=k+j+n-1}^{k+m+n-2} \varphi^{(l+m+n-2)}_w(a)\\
+Xterm
\end{cases}
}& k=i,\\
\varphi^{(l+m+n-2)}_{k+m+n-2}(a), & k>i.
\end{cases}}
\end{aligned}
$$
where
$$\begin{aligned}
Xterm&\underset{\text{by (C37)}}{=}
(\sum\limits_{k\le k_1<k_2\le k+m-1}\psi^{(l+m-1)}_{k_1k_2}(f(a)))
\underset{k-1+j}{\circ} \1_n\\
&={\tiny
\begin{cases}
\sum\limits_{k\le k_1<k_2<k+j-1}\psi^{(l+m+n-2)}_{k_1k_2}(f(a)))\\
+\sum\limits_{k\leq k_1<k+j-1,k+j-1\leq k_2<k+j-1+n-1}
\psi^{(l+m+n-2)}_{k_1k_2}(f(a))\\
+\sum_{k_1<k+j-1<k_2}\psi_{k_1,k_2+n-1}(f(a))\\
+\sum_{k+j-1\leq k_1<k+j-1+n,k+j-1+n\leq k_2<i+m-1+n-1}
\psi^{(l+m+n-2)}_{k_1k_2}(f(a))\\
+\sum\limits_{k\le k+j-1<k_1<k_2}\psi^{(l+m+n-2)}_{k_1+n-1,k_2+n-1}(f(a)))
\end{cases}
}
\end{aligned}
$$
and
$$\begin{aligned}
{\text{RHS of \eqref{E1.1.1}}}
&=\varphi_{k}^{(l)}(a)\underset{i}{\circ} (\1_m \underset{j}{\circ} \1_n)\\
&=\varphi_{k}^{(l)}(a)\underset{i}{\circ} \1_{m+n-1} \\
&\underset{\text{by (C34)}}{=}
{\tiny \begin{cases}
\varphi^{(l+m+n-2)}_k(a), & k<i,\\
\sum\limits_{w=k}^{k+m+n-2}\varphi_{w}^{(l+m+n-2)}(a)
+\sum_{k\leq l_1<l_2\leq k+m+n-2} \psi_{l_1l_2}^{(l+m+n-2)}(f(a))
& k=i,\\
\varphi^{(l+m+n-2)}_{k+m+n-2}(a), & k>i.
\end{cases}
}
\end{aligned}
$$
Hence \eqref{E1.1.1} holds.

{\sf Case 15: $\mu \in \Phi$ and $\lambda,\nu\in {\mathbb A}$. }
Write $\mu=\varphi_{k}^{(m)}(a)$, $\lambda=\1_l$
and $\nu=\1_n$. Then
$$\begin{aligned}
{\text{LHS of \eqref{E1.1.1}}}
&=(\1_l \underset{i}{\circ} \varphi_{k}^{(m)}(a)) \underset{i-1+j}{\circ} \1_n
\underset{\text{by (C32)}}{=}
\varphi_{k+i-1}^{(m+l-1)}(a) \underset{i-1+j}{\circ} \1_n\\
&\underset{\text{by (C34)}}{=}
{\tiny \begin{cases}
\varphi_{k+i-1}^{(m+l+n-2)}(a)
& k<j \\
\sum\limits_{s=k+i-1}^{k+i+n-2} \varphi_{s}^{(m+l+n-2)}(a)
+\sum\limits_{k+i-1\leq k_1<k_2<k+i+n-2} \psi_{k_1k_2}^{(m+l+n-2)}(f(a))
& k=j \\
\varphi_{k+i+n-2}^{(m+l+n-2)}(a)
& k>j
\end{cases}}\\
{\text{RHS of \eqref{E1.1.1}}}
&=\1_l \underset{i}{\circ} (\varphi_{k}^{(m)}(a))\underset{j}{\circ} 1_n)\\
&\underset{\text{by (C34)}}{=}
\1_l \underset{i}{\circ}
{\tiny \begin{cases}
\varphi_{k}^{(m+n-1)}(a)
& k<j \\
\sum\limits_{s=k}^{k+n-1} \varphi_{s}^{(m+n-1)}(a)
+\sum\limits_{k\leq k_1<k_2<k+n-1} \psi_{k_1k_2}^{(m+n-1)}(f(a))
& k=j \\
\varphi_{k+n-1}^{(m+n-1)}(a)
& k>j
\end{cases}}
\\
&\underset{\text{by (C32), (C33)}}{=}
{\tiny \begin{cases}
\varphi_{k+i-1}^{(m+l+n-2)}(a)
& k<j \\
\sum\limits_{s=k+i-1}^{k+i+n-2} \varphi_{s}^{(m+l+n-2)}(a)
+\sum\limits_{k+i-1\leq k_1<k_2<k+i+n-2} \psi_{k_1k_2}^{(m+l+n-2)}(f(a))
& k=j \\
\varphi_{k+i+n-2}^{(m+l+n-2)}(a)
& k>j
\end{cases}}
\end{aligned}
$$
Hence \eqref{E1.1.1} holds.

{\sf Case 16: $\nu \in \Phi$ and $\lambda,\mu\in {\mathbb A}$.}
Write $\nu=\varphi_{k}^{(n)}(a)$, $\lambda=\1_l$
and $\mu=\1_m$. Then
$$\begin{aligned}
{\text{LHS of \eqref{E1.1.1}}}
&=(\1_l \underset{i}{\circ} \1_m) \underset{i-1+j}{\circ} \varphi_{k}^{(n)}(a)
\underset{\text{by (C31)}}{=}
\1_{l+m-1} \underset{i-1+j}{\circ} \varphi_{k}^{(n)}(a)\\
&\underset{\text{by (C32)}}{=}
\varphi_{k+i+j-2}^{(n+l+m-2)}(a)
\\
{\text{RHS of \eqref{E1.1.1}}}
&=\1_l \underset{i}{\circ} (\1_m \underset{j}{\circ} \varphi_{k}^{(n)}(a))
\underset{\text{by (C32)}}{=}
\1_l \underset{i}{\circ} \varphi_{k+j-1}^{(n+m-1)}(a)\\
&\underset{\text{by (C32)}}{=}
\varphi_{k+j+i-2}^{(n+m+l-2)}(a).
\end{aligned}
$$
Hence \eqref{E1.1.1} holds.

{\sf Case 17: $\lambda \in \Psi$ and $\mu,\nu\in {\mathbb A}$. }
Write $\lambda=\psi_{k_1,k_2}^{(l)}(y)$, $\mu=\1_m$
and $\nu=\1_n$. Then
$$\begin{aligned}
{\text{LHS of \eqref{E1.1.1}}}
&=(\psi_{k_1,k_2}^{(l)}(y)\underset{i}{\circ} \1_m) \underset{i-1+j}{\circ} \1_n\\
&\underset{\text{by (C37)}}{=}
{\tiny
\begin{cases}
\psi^{(l+m-1)}_{k_1+m-1, k_2+m-1}(y), & 1\le i<k_1, \\
\sum\limits_{k=k_1}^{k_1+m-1} \psi^{(l+m-1)}_{k, k_2+m-1}(y), & i=k_1, \\
\psi^{(l+m-1)}_{k_1, k_2+m-1}(y), & k_1<i<k_2,\\
\sum\limits_{k=k_2}^{k_2+m-1} \psi^{(l+m-1)}_{k_1, k}(y), & i=k_2,\\
\psi^{(l+m-1)}_{k_1k_2}(y), & k_2<i\le l.
\end{cases}
} \quad \underset{i-1+j}{\circ} \1_n\\
&\underset{\text{by (C37)}}{=}
{\tiny
\begin{cases}
{\tiny
\begin{cases}
\psi^{(l+m+n-2)}_{k_1+m-1+n-1, k_2+m-1+n-1}(y), & 1\le i-1+j<k_1+m-1, \\
(impossible), & others.
\end{cases}
}  & 1\le i<k_1.\\
\sum\limits_{k=k_1}^{k_1+j-2} \psi^{(l+m+n-2)}_{k, k_2+m+n-2}(y)+\sum\limits_{k=k_1-1+j}^{k_1-1+j+n-1} \psi^{(l+m+n-2)}_{k, k_2+m+n-2}(y)+\sum\limits_{k=k_1+j}^{k_1+m-1} \psi^{(l+m+n-2)}_{k+n-1, k_2+m+n-2}(y)& i=k_1.\\
{\tiny
\begin{cases}
\psi^{(l+m+n-2)}_{k_1, k_2+m-1+n-1}(y), & k_1\le i-1+j<k_2+m-1, \\
(impossible), & others.
\end{cases}
}  & k_1\le i<k_2.\\
\sum\limits_{k=k_2}^{k_2+j-2} \psi^{(l+m+n-2)}_{k_1, k}(y)+\sum\limits_{k=k_2-1+j}^{k_2-1+j+n-1} \psi^{(l+m+n-2)}_{k_1, k}(y)+\sum\limits_{k=k_2+j}^{k_2+m-1} \psi^{(l+m+n-2)}_{k_1, k+n-1}(y) & i=k_2.\\
{\tiny
\begin{cases}
\psi^{(l+m+n-2)}_{k_1, k_2}(y), & k_2\le i-1+j<l+m-1, \\
(impossible), & others.
\end{cases}
}  & k_2\le i<l.
\end{cases}}
\end{aligned}
$$
and
$$\begin{aligned}
{\text{RHS of \eqref{E1.1.1}}}
&=\psi_{k_1,k_2}^{(l)}(y)\underset{i}{\circ} (\1_m \underset{j}{\circ} \1_n)\\
&\underset{\text{by (C31)}}{=}\psi_{k_1,k_2}^{(l)}(y)\underset{i}{\circ} \1_{m+n-1} \\
&\underset{\text{by (C34)}}{=}
{\tiny
\begin{cases}
\psi^{(l+m+n-2)}_{k_1+m+n-2, k_2+m+n-2}(y), & 1\le i<k_1, \\
\sum\limits_{k=k_1}^{k_1+m+n-2} \psi^{(m+n-2)}_{k, k_2+m+n-2}(y), & i=k_1, \\
\psi^{(l+m+n-2)}_{k_1, k_2+m+n-2}(y), & k_1<i<k_2,\\
\sum\limits_{k=k_2}^{k_2+m+n-2} \psi^{(l+m+n-2)}_{k_1, k}(y), & i=k_2,\\
\psi^{(l+m+n-2)}_{k_1k_2}(y), & k_2<i\le l.
\end{cases}
}
\end{aligned}
$$
Hence \eqref{E1.1.1} holds.

{\sf Case 18: $\mu \in \Psi$ and $\lambda,\nu\in {\mathbb A}$. }
Write $\mu=\psi_{k_1,k_2}^{(m)}(y)$, $\lambda=\1_l$
and $\nu=\1_n$. Then
$$\begin{aligned}
{\text{LHS of \eqref{E1.1.1}}}
&=(\1_l \underset{i}{\circ} \psi_{k_1,k_2}^{(m)}(y)) \underset{i-1+j}{\circ} \1_n
\underset{\text{by (C32)}}{=}
\psi_{k_1+i-1,k_2+i-1}^{(l+m-1)}(y)\underset{i-1+j}{\circ} \1_n\\
&\underset{\text{by (C37)}}{=}
{\tiny
\begin{cases}
\psi_{k_1+i+n-2,k_2+i+n-2}^{(l+m+n-2)}(y) & j<k_1,\\
\sum\limits_{w=k_1+i-1}^{k_1+i+n-2} \psi_{w,k_2+i+n-1}^{(l+m+n-2)}(y) & j=k_1,\\
\psi_{k_1+i-1,k_2+i+n-2}^{(l+m+n-2)}(y)&k_1<j<k_2,\\
\sum\limits_{w=k_2+i-1}^{k_2+i+n-2} \psi_{k_1+i-1,w}^{(l+m+n-2)}(y) & j=k_2,\\
\psi_{k_1+i-1,k_2+i-1}^{(l+m+n-2)}(y),&j>k_2,
\end{cases}
}
\end{aligned}
$$
and
$$\begin{aligned}
{\text{RHS of \eqref{E1.1.1}}}
&=\1_l \underset{i}{\circ} (\psi_{k_1,k_2}^{(m)}(y)\underset{j}{\circ} \1_n)\\
&\underset{\text{by (C37)}}{=}
\1_l \underset{i}{\circ}
{\tiny
\begin{cases}
\psi^{(m+n-1)}_{k_1+n-1, k_2+n-1}(y), & 1\le j<k_1, \\
\sum\limits_{w=k_1}^{k_1+n-1} \psi^{(m+n-1)}_{w, k_2+n-1}(y), & j=k_1, \\
\psi^{(m+n-1)}_{k_1, k_2+n-1}(y), & k_1<j<k_2,\\
\sum\limits_{w=k_2}^{k_2+n-1} \psi^{(m+n-1)}_{k_1, w}(y), & j=k_2,\\
\psi^{(m+n-1)}_{k_1k_2}(y), & k_2<j\le m
\end{cases}
}\\
&\underset{\text{by (C33)}}{=}
{\tiny
\begin{cases}
\psi^{(l+m+n-1)}_{k_1+i+n-2, k_2+i+n-2}(y), & 1\le j<k_1, \\
\sum\limits_{w=k_1+i-1}^{k_1+i+n-2} \psi^{(m+n-1)}_{w, k_2+i+n-2}(y), & j=k_1, \\
\psi^{(l+m+n-1)}_{k_1+i-1, k_2+i+n-2}(y), & k_1<j<k_2,\\
\sum\limits_{w=k_2+i-1}^{k_2+i+n-2} \psi^{(m+n-1)}_{k_1+i-1, w}(y), & j=k_2,\\
\psi^{(l+m+n-1)}_{k_1+i-1,k_2+i-1}(y), & k_2<j\le m.
\end{cases}
}
\end{aligned}
$$
Hence \eqref{E1.1.1} holds.

{\sf Case 19: $\nu \in \Psi$ and $\lambda,\mu\in {\mathbb A}$.}
Write $\nu=\psi_{k_1,k_2}^{(n)}(y)$, $\lambda=\1_l$
and $\mu=\1_m$. Then
$$\begin{aligned}
{\text{LHS of \eqref{E1.1.1}}}
&=(\1_l \underset{i}{\circ} \1_m) \underset{i-1+j}{\circ} \psi_{k_1,k_2}^{(n)}(y)
\underset{\text{by (C31)}}{=}
\1_{l+m-1} \underset{i-1+j}{\circ} \psi_{k_1,k_2}^{(n)}(y)\\
&\underset{\text{by (C33)}}{=}
\psi_{k_1+i+j-2,k_2+i+j-2}^{(n+l+m-2)}(y)  \\
{\text{RHS of \eqref{E1.1.1}}}
&\underset{\text{by (C33)}}{=}
\1_l \underset{i}{\circ} (\1_m \underset{j}{\circ} \psi_{k_1,k_2}^{(n)}(y)
=\1_l \underset{i}{\circ} \psi_{k_1+j-1,k_2+j-1}^{(n+m-1)}(y)\\
&\underset{\text{by (C33)}}{=}
\psi_{k_1+i+j-2,k_2+i+j-2}^{(n+l+m-2)}(y)
\end{aligned}
$$
Hence \eqref{E1.1.1} holds.

{\sf Case 20: $\lambda,\mu,\nu\in {\mathbb A}$.} Equation
\eqref{E1.1.1} follows from (C31).

\noindent
{\it Verification of (OP3):} There are two equations in (OP3).
We only check the first one in (OP3), namely, the following
equation
\begin{equation}
\label{E1.1.3a}\tag{E1.1.3}
\mu  \ucr{i} (\nu \ast \sigma)=(\mu  \ucr{i} \nu)\ast \sigma',
\end{equation}
for all $\mu\in \ip(m), \nu\in \ip(n)$, and $\sigma\in \S_n$.
Note that $\sigma'$ is $1_m\ucr{i} \sigma$. Write
\begin{equation}
\label{E3.6.1}\tag{E3.6.1}
\sigma=\begin{pmatrix} k_1, & k_2, &\cdots & k_n
\end{pmatrix}
\end{equation}
where by convention $k_w=\sigma^{-1}(w)$ for all $w$.
Then, by definition,
\begin{equation}
\label{E3.6.2}\tag{E3.6.2}
{\small \sigma'=\begin{pmatrix} 1, &\cdots &i-1, & k_1+i-1,
& k_2+i-1, &\cdots & k_n+i-1, &
i+n,&\cdots &n+m-1\end{pmatrix}.}
\end{equation}
By \eqref{E3.6.2}, $(\sigma')^{-1}(s)=s$ if $s<i$ and $s\geq i+n$
and $(\sigma')^{-1}(k+i-1)=\sigma^{-1}(k)+i-1$ for $1\leq k \leq n$.
We refer to \cite[Section 8]{BYZ} for more details concerning $\sigma'$.

If both $\mu$ and $\nu$ are in ${\mathbb A}$, then \eqref{E1.1.3}
follows easily from (OP1). We have the following 8 cases to
consider.

{\sf Case 1: $\mu \in \Phi$ and $\nu\in {\mathbb A}$. }
Write $\mu=\varphi^{(m)}_k(a)$ and $\nu=\1_n$. Then
$$\begin{aligned}
{\text{LHS of \eqref{E1.1.3}}}
&=\varphi^{(m)}_k(a) \underset{i}{\circ} (\1_n\ast \sigma)
\underset{\text{by (C33)}}{=}
\varphi^{(m)}_k(a) \underset{i}{\circ}\1_n\\
&\underset{\text{by (C34)}}{=}
{\tiny
\begin{cases}
\varphi^{(m+n-1)}_k(a), & k<i,\\
\sum\limits_{w=k}^{k+n-1}\varphi^{(m+n-1)}_{k}(a)
+\sum\limits_{k\le k_1<k_2\le k+n-1}\psi^{(m+n-1)}_{k_1k_2}(f(a)), & k=i, \\
\varphi^{(m+n-1)}_{k+n-1}(a), & k>i.
\end{cases}
}
\end{aligned}
$$
and in the following computation we use the fact that
$f(a)\ast (2,1)=f(a)$ [Definition \ref{xxdef3.4}(3)] and
notation $(k'_1,k'_2)=((\sigma')^{-1}(k_1),(\sigma')^{-1}(k_2))$
or $((\sigma')^{-1}(k_2),(\sigma')^{-1}(k_1))$,
$$\begin{aligned}
&{\text{RHS of \eqref{E1.1.3}}}
=(\varphi^{(m)}_k(a) \underset{i}{\circ}\1_n)\ast \sigma'\\
&\underset{\text{by (C34)}}{=}
{\tiny
\begin{cases}
\varphi^{(m+n-1)}_k(a), & k<i,\\
\sum\limits_{w=k}^{k+n-1}\varphi^{(m+n-1)}_{k}(a)
+\sum\limits_{k\le k_1<k_2\le k+n-1}\psi^{(m+n-1)}_{k_1k_2}(f(a)), & k=i, \\
\varphi^{(m+n-1)}_{k+n-1}(a), & k>i.
\end{cases}
}\quad \ast \sigma'\\
&\underset{\text{by (C22) and (C23)}}{=}
{\tiny
\begin{cases}
\varphi^{(m+n-1)}_{(\sigma')^{-1}(k)}(a), & k<i,\\
\sum\limits_{w=k}^{k+n-1}\varphi^{(m+n-1)}_{(\sigma')^{-1}(k)}(a)
+\sum\limits_{k\le k'_1<k'_2\le k+n-1}
\psi^{(m+n-1)}_{k'_1k'_2}(f(a)), & k=i, \\
\varphi^{(m+n-1)}_{(\sigma')^{-1}(k+n-1)}(a), & k>i.
\end{cases}
}\\
&={\tiny
\begin{cases}
\varphi^{(m+n-1)}_{k}(a), & k<i,\\
\sum\limits_{w=k}^{k+n-1}\varphi^{(m+n-1)}_{k}(a)
+\sum\limits_{k\le k_1<k_2\le k+n-1}
\psi^{(m+n-1)}_{k_1k_2}(f(a)), & k=i, \\
\varphi^{(m+n-1)}_{k+n-1}(a), & k>i.
\end{cases}
}\\
\end{aligned}
$$
Hence \eqref{E1.1.3} holds.

{\sf Case 2: $\mu \in \Psi$ and $\nu\in {\mathbb A}$. }
Write $\mu=\psi^{(m)}_{k_1k_2}(y)$ and $\nu=\1_n$. Then
$$\begin{aligned}
{\text{LHS of \eqref{E1.1.3}}}
&=\psi^{(m)}_{k_1k_2}(y) \underset{i}{\circ} (\1_n\ast \sigma)
\underset{\text{by (C33)}}{=}
\psi^{(m)}_{k_1k_2}(y) \underset{i}{\circ}\1_n\\
&\underset{\text{by (C37)}}{=}
{\tiny \begin{cases}
\psi^{(m+n-1)}_{k_1+n-1, k_2+n-1}(y), & 1\le i<k_1, \\
\sum\limits_{w=k_1}^{k_1+n-1} \psi^{(m+n-1)}_{w, k_2+n-1}(y), & i=k_1, \\
\psi^{(m+n-1)}_{k_1, k_2+n-1}(y), & k_1<i<k_2,\\
\sum\limits_{w=k_2}^{k_2+n-1} \psi^{(m+n-1)}_{k_1, w}(y), & i=k_2,\\
\psi^{(m+n-1)}_{k_1k_2}(y), & k_2<i\le m,
\end{cases}
}
\end{aligned}
$$
and
$$\begin{aligned}
{\text{RHS of \eqref{E1.1.3}}}
&=(\psi^{(m)}_{k_1k_2}(y) \underset{i}{\circ} \1_n)\ast \sigma'\\
&\underset{\text{by (C37)}}{=}
{\tiny \begin{cases}
\psi^{(m+n-1)}_{k_1+n-1, k_2+n-1}(y), & 1\le i<k_1, \\
\sum\limits_{w=k_1}^{k_1+n-1} \psi^{(m+n-1)}_{w, k_2+n-1}(y), & i=k_1, \\
\psi^{(m+n-1)}_{k_1, k_2+n-1}(y), & k_1<i<k_2,\\
\sum\limits_{w=k_2}^{k_2+n-1} \psi^{(m+n-1)}_{k_1, w}(y), & i=k_2,\\
\psi^{(m+n-1)}_{k_1k_2}(y), & k_2<i\le m,
\end{cases}
} \quad \ast \sigma'\\
&\underset{\text{by \eqref{E3.6.2}}}{=}
{\tiny \begin{cases}
\psi^{(m+n-1)}_{k_1+n-1, k_2+n-1}(y), & 1\le i<k_1, \\
\sum\limits_{w=k_1}^{k_1+n-1} \psi^{(m+n-1)}_{w, k_2+n-1}(y), & i=k_1, \\
\psi^{(m+n-1)}_{k_1, k_2+n-1}(y), & k_1<i<k_2,\\
\sum\limits_{w=k_2}^{k_2+n-1} \psi^{(m+n-1)}_{k_1, w}(y), & i=k_2,\\
\psi^{(m+n-1)}_{k_1k_2}(y), & k_2<i\le m.
\end{cases}
}
\end{aligned}
$$
Hence \eqref{E1.1.3} holds.

{\sf Case 3: $\mu \in {\mathbb A}$ and $\nu\in \Phi$. }
Write $\mu=\1_m$ and $\nu=\varphi^{(n)}_k(a)$. Then
$$\begin{aligned}
{\text{LHS of \eqref{E1.1.3}}}
&=\1_m \underset{i}{\circ} (\varphi^{(n)}_k(a)\ast \sigma)
\underset{\text{by (C22)}}{=}
\1_m \underset{i}{\circ} \varphi^{(n)}_{\sigma^{-1}(k)}(a)\\
&\underset{\text{by (C32)}}{=}
\varphi^{(m+n-1)}_{\sigma^{-1}(k)+i-1}(a),
\end{aligned}
$$
and
$$\begin{aligned}
{\text{RHS of \eqref{E1.1.3}}}
&=(\1_m \underset{i}{\circ} \varphi^{(n)}_k(a))\ast \sigma'
\underset{\text{by (C32)}}{=}
\varphi^{(m+n-1)}_{k+i-1}(a) \ast \sigma'\\
&\underset{\text{by (C22)}}{=}
\varphi^{(m+n-1)}_{(\sigma')^{-1}(k+i-1)}(a)
=\varphi^{(m+n-1)}_{\sigma^{-1}(k)+i-1}(a).
\end{aligned}
$$
Hence \eqref{E1.1.3} holds.

{\sf Case 4: $\mu \in {\mathbb A}$ and $\nu\in \Psi$. }
Write $\mu=\1_m$ and $\nu=\psi^{(n)}_{k_1k_2}(y)$. Then
$$\begin{aligned}
{\text{LHS of \eqref{E1.1.3}}}
&=\1_m \underset{i}{\circ} (\psi^{(n)}_{k_1k_2}(y)\ast \sigma)
\underset{\text{by (C23)}}{=}\1_m \underset{i}{\circ}
{\tiny
\begin{cases}
\psi^{(n)}_{\sigma^{-1}(k_1)\sigma^{-1}(k_2)}(y) &
\sigma^{-1}(k_1)<\sigma^{-1}(k_2)\\
\psi^{(n)}_{\sigma^{-1}(k_2)\sigma^{-1}(k_1)}(y\ast (2,1)) &
\sigma^{-1}(k_1)>\sigma^{-1}(k_2)
\end{cases}
}\\
&\underset{\text{by (C33)}}{=}{\tiny
\begin{cases}
\psi^{(m+n-1)}_{\sigma^{-1}(k_1)+i-1,\sigma^{-1}(k_2)+i-1}(y) &
\sigma^{-1}(k_1)<\sigma^{-1}(k_2)\\
\psi^{(m+n-1)}_{\sigma^{-1}(k_2)+i-1,\sigma^{-1}(k_1)+i-1}(y\ast (2,1)) &
\sigma^{-1}(k_1)>\sigma^{-1}(k_2)
\end{cases}
}
\end{aligned}
$$
and
$$\begin{aligned}
{\text{RHS of \eqref{E1.1.3}}}
&=(\1_m \underset{i}{\circ} \psi^{(n)}_{k_1k_2}(y))\ast \sigma'\\
&\underset{\text{by (C33)}}{=}
\psi^{(m+n-1)}_{k_1+i-1,k_2+i-1}(y)\ast \sigma'\\
&\underset{\text{by (C23)}}{=}{\tiny
\begin{cases}
\psi^{(m+n-1)}_{(\sigma')^{-1}(k_1+i-1),(\sigma')^{-1}(k_2+i-1)}(y) &
(\sigma')^{-1}(k_1+i-1)<(\sigma')^{-1}(k_2+i-1)\\
\psi^{(m+n-1)}_{(\sigma')^{-1}(k_2+i-1),(\sigma')^{-1}(k_1+i-1)}(y) &
(\sigma')^{-1}(k_1+i-1)>(\sigma')^{-1}(k_2+i-1)\\
\end{cases}
}\\
&={\tiny
\begin{cases}
\psi^{(m+n-1)}_{\sigma^{-1}(k_1)+i-1,\sigma^{-1}(k_2)+i-1}(y) &
\sigma^{-1}(k_1)<\sigma^{-1}(k_2)\\
\psi^{(m+n-1)}_{\sigma^{-1}(k_2)+i-1,\sigma^{-1}(k_1)+i-1}(y\ast (2,1)) &
\sigma^{-1}(k_1)>\sigma^{-1}(k_2).
\end{cases}
}
\end{aligned}
$$
Hence \eqref{E1.1.3} holds.

{\sf Case 5: $\mu \in \Phi$ and $\nu\in \Phi$. }
Write $\mu=\varphi^{(m)}_{k_1}(a)$ and
$\nu=\varphi^{(n)}_{k_2}(b)$. Then
$$\begin{aligned}
{\text{LHS of \eqref{E1.1.3}}}
&=\varphi^{(m)}_{k_1}(a)\underset{i}{\circ}
(\varphi^{(n)}_{k_2}(b)\ast \sigma)
\underset{\text{by (C22)}}{=}
\varphi^{(m)}_{k_1}(a)\underset{i}{\circ}
\varphi^{(n)}_{\sigma^{-1}(k_2)}(b)\\
&\underset{\text{by (C35)}}{=}
{\tiny \begin{cases}
\psi^{(m+n-1)}_{k_1, i+\sigma^{-1}(k_2)-1}(g(a, b)), & k_1<i, \\
\begin{array}{ll}
\sum\limits_{k=k_1}^{k_1+\sigma^{-1}(k_2)-2} \psi^{(m+n-1)}_{k, k_1+\sigma^{-1}(k_2)-1}(g(a, b)
+f(a)\underset{2}\cdot b)+\varphi^{(m+n-1)}_{k_1+\sigma^{-1}(k_2)-1}(ab)\\
+\sum\limits_{k=k_1+\sigma^{-1}(k_2)}^{k_1+n-1} \psi^{(m+n-1)}_{k_1+\sigma^{-1}(k_2)-1, k}(g(b, a)
+f(a)\underset{1}\cdot b),
\end{array}& k_1=i, \\
\psi^{(m+n-1)}_{i+\sigma^{-1}(k_2)-1, k_1+n-1}(g(b, a)), & k_1>i,
\end{cases}}
\end{aligned}
$$
and
$$\begin{aligned}
{\text{RHS of \eqref{E1.1.3}}}
&=(\varphi^{(m)}_{k_1}(a)\underset{i}{\circ}
\varphi^{(n)}_{k_2}(b))\ast \sigma'\\
&\underset{\text{by (C35)}}{=}
{\tiny \begin{cases}
\psi^{(m+n-1)}_{k_1, i+k_2-1}(g(a, b)), & k_1<i, \\
\begin{array}{ll}
\sum\limits_{k=k_1}^{k_1+k_2-2} \psi^{(m+n-1)}_{k, k_1+k_2-1}(g(a, b)
+f(a)\underset{2}\cdot b)+\varphi^{(m+n-1)}_{k_1+k_2-1}(ab)\\
+\sum\limits_{k=k_1+k_2}^{k_1+n-1} \psi^{(m+n-1)}_{k_1+k_2-1, k}(g(b, a)
+f(a)\underset{1}\cdot b),
\end{array}& k_1=i, \\
\psi^{(m+n-1)}_{i+k_2-1, k_1+n-1}(g(b, a)), & k_1>i.
\end{cases}}\quad \ast \sigma'\\
&\underset{\text{by (C22) and (C23)}}{=}
{\tiny \begin{cases}
\psi^{(m+n-1)}_{k_1, i+\sigma^{-1}(k_2)-1}(g(a, b)), & k_1<i, \\
\begin{array}{ll}
\sum\limits_{k=k_1}^{k_1+\sigma^{-1}(k_2)-2} \psi^{(m+n-1)}_{k, k_1+\sigma^{-1}(k_2)-1}(g(a, b)
+f(a)\underset{2}\cdot b)+\varphi^{(m+n-1)}_{k_1+\sigma^{-1}(k_2)-1}(ab)\\
+\sum\limits_{k=k_1+\sigma^{-1}(k_2)}^{k_1+n-1} \psi^{(m+n-1)}_{k_1+\sigma^{-1}(k_2)-1, k}(g(b, a)
+f(a)\underset{1}\cdot b),
\end{array}& k_1=i, \\
\psi^{(m+n-1)}_{i+\sigma^{-1}(k_2)-1, k_1+n-1}(g(b, a)), & k_1>i.
\end{cases}}
\end{aligned}
$$
Hence \eqref{E1.1.3} holds.

{\sf Case 6: $\mu \in \Phi$ and $\nu\in \Psi$. }
Write $\mu=\varphi^{(m)}_{h}(a)$ and
$\nu=\psi^{(n)}_{k_1k_2}(y)$. Then
$$\begin{aligned}
{\text{LHS of \eqref{E1.1.3}}}
&=\varphi^{(m)}_{h}(a)\underset{i}{\circ}
(\psi^{(n)}_{k_1k_2}(y)\ast \sigma)
\underset{\text{by (C23)}}{=}
\varphi^{(m)}_{h}(a)\underset{i}{\circ}
\psi^{(n)}_{\sigma^{-1}(k_1)\sigma^{-1}(k_2)}(y)\\
&\underset{\text{by (C36)}}{=}
\begin{cases}
0, & h\neq i, \\
\psi^{(m+n-1)}_{h+\sigma^{-1}(k_1)-1, h+\sigma^{-1}(k_1)-1}(ay), & h=i,\\
\end{cases}
\end{aligned}$$
and
$$\begin{aligned}
{\text{RHS of \eqref{E1.1.3}}}
&=(\varphi^{(m)}_{h}(a)\underset{i}{\circ}
\psi^{(n)}_{k_1k_2}(y))\ast \sigma'\\
&\underset{\text{by (C36)}}{=}
\begin{cases}
0, & h\neq i, \\
\psi^{(m+n-1)}_{h+k_1-1, h+k_1-1}(ay), & h=i,\\
\end{cases} \quad \ast \sigma'\\
&\underset{\text{by (C23)}}{=}
\begin{cases}
0, & h\neq i, \\
\psi^{(m+n-1)}_{h+\sigma^{-1}(k_1)-1, h+\sigma^{-1}(k_1)-1}(ay), & h=i.\\
\end{cases}
\end{aligned}$$
Hence \eqref{E1.1.3} holds.

{\sf Case 7: $\mu \in \Psi$ and $\nu\in \Phi$. }
Write $\mu=\psi^{(m)}_{k_1k_2}(y)$ and
$\nu=\varphi^{(n)}_{h}(a)$. Then
$$\begin{aligned}
{\text{LHS of \eqref{E1.1.3}}}
&=\psi^{(m)}_{k_1k_2}(y)\underset{i}{\circ}
(\varphi^{(n)}_{h}(a)\ast \sigma)
\underset{\text{by (C22)}}{=}
\psi^{(m)}_{k_1k_2}(y)\underset{i}{\circ}
\varphi^{(n)}_{\sigma^{-1}(h)}(a)\\
&\underset{\text{by (C38)}}{=}
{\tiny \begin{cases}
0, & i\neq k_1, k_2, \\
\psi^{(m+n-1)}_{k_1+\sigma^{-1}(h)-1, k_2+n-1}(y \underset{1}\cdot a), & i=k_1, \\
\psi^{(m+n-1)}_{k_1, k_2+\sigma^{-1}(h)-1}(y \underset{2}\cdot a), & i=k_2,\\
\end{cases}
}
\end{aligned}
$$
and
$$
\begin{aligned}
{\text{RHS of \eqref{E1.1.3}}}
&=(\psi^{(m)}_{k_1k_2}(y)\underset{i}{\circ}
\varphi^{(n)}_{h}(a))\ast \sigma'\\
&\underset{\text{by (C38)}}{=}
{\tiny \begin{cases}
0, & i\neq k_1, k_2, \\
\psi^{(m+n-1)}_{k_1+h-1, k_2+n-1}(y \underset{1}\cdot a), & i=k_1, \\
\psi^{(m+n-1)}_{k_1, k_2+h-1}(y \underset{2}\cdot a), & i=k_2.\\
\end{cases}
} \quad \ast \sigma'\\
&\underset{\text{by (C23)}}{=}
{\tiny \begin{cases}
0, & i\neq k_1, k_2, \\
\psi^{(m+n-1)}_{k_1+\sigma^{-1}(h)-1, k_2+n-1}(y \underset{1}\cdot a), & i=k_1, \\
\psi^{(m+n-1)}_{k_1, k_2+\sigma^{-1}(h)-1}(y \underset{2}\cdot a), & i=k_2.\\
\end{cases}
}
\end{aligned}
$$
Hence \eqref{E1.1.3} holds.

{\sf Case 8: $\mu \in \Psi$ and $\nu\in \Psi$. }
In this case both sides of \eqref{E1.1.3} are zero, so
\eqref{E1.1.3} holds.

Combining all these cases, $\ip$ is a 2-unitary operad.
We define a morphism $u_{\ip}:\Com\to \ip$ by sending
$\1_m\to \1_m$ for all $m\geq 0$. Note that $u_{\ip}$
is an operadic morphism by (C21) and (C31).

Let $\{a_j\}_{j=1}^{d}$ be a basis of ${\bar A}$ and
$\{\mu_k\}_{k=1}^{m}$ a basis of $M$ where $d$ is the
dimension of $\bar A$ and $m$ is the dimension of $M$.
By construction,
$$\{\1_n, \varphi^{(n)}_{i}(a_j):=\1_n\underset{i}\circ a_j,
\psi^{(n)}_{i_1i_2}(\mu_k):=(\1_{n-1}\underset{1}\circ \mu_k)
\ast c_{i_1i_2}\mid i\in [n],j\in [d],k\in [m], 1\leq i_1<i_2\leq n\}$$
is a $\Bbbk$-basis of $\ip(n)$. As a consequence, the generating
function of $\ip$ is
\[G_{\ip}(t)=\sum\limits_{n=0}^{\infty}(1+dn+m\frac{n(n-1)}{2})t^n
=\frac{1}{1-t}+d\frac{t}{(1-t)^2}+m\frac{t^2}{(1-t)^3}.\]
Therefore $\ip$ has GK-dimension 3.
\end{proof}

\end{document}